\documentclass[a4, 12pt, reqno]{amsart}
\usepackage{amsmath}
\usepackage{amssymb}
\usepackage{amsthm} 
\usepackage{tikz-cd}
\usepackage[colorlinks=true,linkcolor=blue,citecolor=purple]{hyperref}
\makeatletter
\newcommand{\arxiv}[1]{\href{http://arxiv.org/abs/#1}{\tt
    arXiv:\nolinkurl{#1}}}
\tikzset{
	ch/.style={circle,draw,on chain,inner sep=2pt},
	chj/.style={ch,join},
	every path/.style={shorten >=4pt,shorten <=4pt}
}

\newcommand{\RomanNumeralCaps}[1]
{\MakeUppercase{\romannumeral #1}}

\theoremstyle{plain}
\newtheorem{thm}{Theorem}[section]
\newtheorem{lemma}[thm]{Lemma}
\newtheorem{prop}[thm]{Proposition}
\theoremstyle{definition}
\newtheorem{df}[thm]{Definition}
\numberwithin{equation}{section}
\allowdisplaybreaks
\def\bfU{\mathbf{U}}

\title[Differential operator realization of braid group action]{Differential operator realization of braid group action on $\imath$quantum groups}

\author[Zhaobing Fan]{Zhaobing Fan}
\address{Harbin Engineering University, Harbin, China}
\email{fanzhaobing@hrbeu.edu.cn}
\thanks{}

\author[Jicheng Geng]{Jicheng Geng}
\address{Harbin Engineering University, Harbin, China}
\email{jcgeng@hrbeu.edu.cn}
\thanks{}

\author[Shaolong Han]{Shaolong Han}
\address{Harbin Engineering University, Harbin, China}
\email{algebra@hrbeu.edu.cn}
\thanks{}

\keywords{$\imath$quantum group, modified $q$-Weyl algebra, braid group action.}

\begin{document}

\begin{abstract}
We construct a unique braid group action on modified $q$-Weyl algebra $\mathbf A_q(S)$. Under this action, we give a realization of the braid group action on quasi-split $\imath$quantum groups $^{\imath}\bfU(S)$ of type $\mathrm{AIII}$. Furthermore, we  directly construct a unique braid group action on polynomial ring $\mathbb P$ which is compatible with the braid group action on $\mathbf A_q(S)$ and $^{\imath}\bfU(S)$.
\end{abstract}

\maketitle

\setcounter{tocdepth}{2}

\section{Introduction}
The quantum groups, introduced independently by Drinfeld and Jimbo (cf. \cite{Dr86,J86}), are certain families of  Hopf algebras which are deformations of universal enveloping algebras of Kac-Moody algebras. 
In the theory of quantum groups, braid group action is an important research topic. These actions lead to the construction of PBW bases and canonical bases of quantum groups (cf. \cite{Lus90a, Lus90b, Lus94}). 
When we focus on affine quantum groups, besides their Serre presentation, the Drinfeld presentation is another important presentation (cf. \cite{Dr88}), which plays a fundamental role in the representation theory of affine quantum groups.
The braid group action of affine quantum groups is an essential tool to study Drinfeld’s new realization (cf. \cite{Da93,Be94}).
Moreover, the braid group action can be applied widely to the field of  geometric representation theory, knot theory and  categorification.
\vskip 2mm
The $\imath$quantum groups $\bf U^{\imath}$ arising from quantum symmetric pairs $(\bf U,\bf U^{\imath})$ associated to Satake diagrams (cf. \cite{Let99,Kolb14}) can be seen as non-trivial generalization of quantum groups. 
Therefore, it is natural to ask whether there exists braid group action on $\imath$quantum groups. In \cite{KP11}, Kolb and Pellegrini used software to construct the braid group action on a class of $\imath$quantum groups of finite types (including all quasi-split types and type AII). 
The braid group of type AIII/AIV was constructed by Dobson in \cite{Dob20}. The Hall algebraic approach to braid group action of quasi-split $\imath$quantum groups was studied in \cite{LW21}. For general case, Wang and Zhang gave a conceptual construction of relative braid group action on $\imath$quantum groups of arbitrary finite types (cf. \cite{WZ22}). 
Soon after, Zhang generalized these constructions to Kac-Moody type (cf. \cite{Z22}). The Drinfeld type presentation for $\imath$quantum groups can be found in a series of papers \cite{LW21b, LWZ22,Z21}.
\vskip 2mm

The aim of this paper is to give differential operator realization of braid group action on two typical quasi-split $\imath$quantum groups of type AIII, i.e., $\bfU^\jmath$ and $\bfU^\imath$. In \cite{FGH}, the authors introduced the modified $q$-Weyl algebra $\mathbf A_q(\mathcal{S})$ associated with Satake diagram $\mathcal S$ by using modified $q$-differential operators on polynomial ring. They constructed an algebraic homomorphism $\varphi$ from quasi-split $\imath$quantum groups $^\imath\bfU(\mathcal{S})$ to $\mathbf A_q(\mathcal{S})$. Within this framework, we shall construct braid group action of type $B$ on two classes of modified $q$-Weyl  algebras. Subsequently, we show that the braid
group operators on $\mathbf A_q(\mathcal{S})$ are intertwined with Kolb's braid
group operators via the homomorphism $\varphi$. This means that we give a realization of Kolb's braid group action. In this realization, we show that the braid group action of $\mathbf A_q(\mathcal{S})$  is unique. It is worth mentioning that this realization can be seen as the generalization of Floreanini and Vinet's work (cf. \cite{FV91}). They defined the braid group action on $q$-Weyl algebra $W_q(n)$ which is intertwined with the braid group action of quantum groups $U_q(A_{n-1})$ and $U_q(C_n)$ with respect to the oscillator representations.
\vskip 2mm
Since the modified $q$-Weyl algebra $\mathbf A_q(\mathcal{S})$ has a natural action on a polynomial ring $\mathbb P$, the $^\imath\bfU(\mathcal{S})$-module structure on $\mathbb P$ can be obtained naturally by the homomorphism $\varphi$. Using the braid group action on $\mathbf A_q(\mathcal{S})$, we further construct directly braid group action on $\mathbb P$ which is compatible with the braid group action on $^\imath\bfU(\mathcal{S})$. We also show that this compatible braid group action on $\mathbb P$ is unique. This is different from the approach by Wang and Zhang's braid group action on $^\imath\bfU(\mathcal{S})$-modules (cf. \cite[Section 10.2]{WZ22}), because they regarded the $\mathbf U$-module $M$ as a $^\imath\bfU(\mathcal{S})$-module by restriction.
 \vskip 2mm
Motivated by this work, in the forthcoming paper, we shall study the braid group action on the modified $q$-weyl algebra corresponding to affine quasi-split $\imath$quantum groups, and then obtain the braid group action on these affine $\imath$quantum groups.
 \vskip 2mm
The paper is organized as follows. In Section \ref{sec:iquantum}, we recall the definition of $\imath$quantum groups $^{\imath}\bfU(\mathcal{S})$ associated to two Satake diagrams and the braid group action on these $\imath$quantum groups. 
In Section \ref{sec:qWeyl}, we define four variants braid group operators $T_{i,e}'$ and $T_{i,-e}''$ $(e\in\{1,-1\})$ on modified $q$-Weyl algebra $\mathbf A_q(\mathcal{S})$ and show that they satisfy  braid group relations of type $B$. Moreover, we show that braid group action on $\mathbf A_q(\mathcal{S})$ coincides with the action on $^{\imath}\bfU(\mathcal{S})$ by the homomorphism $\varphi:{^{\imath}\bfU(\mathcal{S})}\to \mathbf A_q(\mathcal{S})$. In Section \ref{sec:polynomial},  we define the braid group action on polynomial ring $\mathbb P$ and further prove that the braid group action on $\mathbf A_q(\mathcal{S})$  is compatible with the action on  $\mathbb P$. This compatible action induce compatible braid group action of $^{\imath}\bfU(\mathcal{S})$ on $\mathbb P$.

\vskip 3mm

{\bf Acknowledgements}.
Z. Fan was partially supported by the NSF of China grant 12271120, the NSF of Heilongjiang Province grant JQ2020A001, and the Fundamental Research Funds for the central universities.

\section{Braid group action on $\imath$quantum groups}\label{sec:iquantum}
We consider the Dynkin diagram of type $A_{n-1}$. There are $n$ nodes $\{1,2,\cdots,n\}$ on this Dynkin diagram. We define an involution $\rho_n$ between these nodes by $\rho_n(i)=n+1-i$ $(1\leq i\leq n)$.
Let $n=2r$ or $2r+1$. The  definition of $\rho_n$ can be described as the following Satake diagrams.
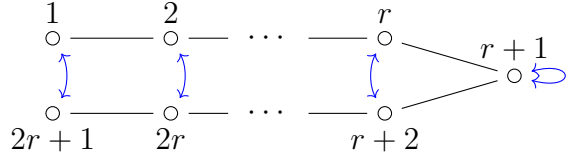
\begin{center}	
	\label{figure:j}	
	\begin{tikzpicture}
		\coordinate (A) at (0,1.5);
		\node at (A) {Satake diagram \RomanNumeralCaps{1}
		};
		\matrix [column sep={0.6cm}, row sep={0.5 cm,between origins}, nodes={draw = none,  inner sep = 3pt}]
		{
			\node(U1) [draw, circle, fill=white, scale=0.6, label = 1] {};
			&\node(U2)[draw, circle, fill=white, scale=0.6, label =2] {};
			&\node(U3) {$\cdots$};
			&\node(U4)[draw, circle, fill=white, scale=0.6, label =$r-1$] {};
			&\node(U5)[draw, circle, fill=white, scale=0.6, label =$r$] {};
			\\
			&&&&&
			\\
			\node(L1) [draw, circle, fill=white, scale=0.6, label =below:$2r$] {};
			&\node(L2)[draw, circle, fill=white, scale=0.6, label =below:$2r-1$] {};
			&\node(L3) {$\cdots$};
			&\node(L4)[draw, circle, fill=white, scale=0.6, label =below:$r+2$] {};
			&\node(L5)[draw, circle, fill=white, scale=0.6, label =below:$r+1$] {};
			\\
		};
		\begin{scope}
			
			\draw (U1) -- node  {} (U2);
			\draw (U2) -- node  {} (U3);
			\draw (U3) -- node  {} (U4);
			\draw (U4) -- node  {} (U5);
			\draw (U5) -- node  {} (L5);
			\draw (L1) -- node  {} (L2);
			\draw (L2) -- node  {} (L3);
			\draw (L3) -- node  {} (L4);
			\draw (L4) -- node  {} (L5);
			\draw (L1) edge [color = blue,<->, bend right, shorten >=4pt, shorten <=4pt] node  {} (U1);
			\draw (L2) edge [color = blue,<->, bend right, shorten >=4pt, shorten <=4pt] node  {} (U2);
			\draw (L4) edge [color = blue,<->, bend left, shorten >=4pt, shorten <=4pt] node  {} (U4);
			\draw (L5) edge [color = blue,<->, bend left, shorten >=4pt, shorten <=4pt] node  {} (U5);
		\end{scope}
	\end{tikzpicture}
\end{center}
\vskip 3mm
\begin{center}
	\label{figure:i}
	\begin{tikzpicture}
		\coordinate (A) at (0,1.5);
		\node at (A) {Satake diagram \RomanNumeralCaps{2}};		
		\matrix [column sep={0.6cm}, row sep={0.5 cm,between origins}, nodes={draw = none,  inner sep = 3pt}]
		{
			\node(U1) [draw, circle, fill=white, scale=0.6, label = 1] {};
			&\node(U2)[draw, circle, fill=white, scale=0.6, label =2] {};
			&\node(U3) {$\cdots$};
			&\node(U5)[draw, circle, fill=white, scale=0.6, label =$r$] {};
			\\
			&&&&
			\node(R)[draw, circle, fill=white, scale=0.6, label =$r+1$] {};
			\\
			\node(L1) [draw, circle, fill=white, scale=0.6, label =below:$2r+1$] {};
			&\node(L2)[draw, circle, fill=white, scale=0.6, label =below:$2r$] {};
			&\node(L3) {$\cdots$};
			&\node(L5)[draw, circle, fill=white, scale=0.6, label =below:$r+2$] {};
			\\
		};
		\begin{scope}
			\draw (U1) -- node  {} (U2);
			\draw (U2) -- node  {} (U3);
			\draw (U3) -- node  {} (U5);
			\draw (U5) -- node  {} (R);
			
			\draw (L1) -- node  {} (L2);
			\draw (L2) -- node  {} (L3);
			\draw (L3) -- node  {} (L5);
			\draw (L5) -- node  {} (R);
			\draw (R) edge [color = blue,loop right, looseness=40, <->, shorten >=4pt, shorten <=4pt] node {} (R);
			\draw (L1) edge [color = blue,<->, bend right, shorten >=4pt, shorten <=4pt] node  {} (U1);
			\draw (L2) edge [color = blue,<->, bend right, shorten >=4pt, shorten <=4pt] node  {} (U2);
			\draw (L5) edge [color = blue,<->, bend left, shorten >=4pt, shorten <=4pt] node  {} (U5);
		\end{scope}
	\end{tikzpicture}
\end{center}

Let $I=\{1,2,\cdots,n\}$. We denote by $\mathcal S$ the above Satake diagrams uniformly.
\begin{df}\cite[Proposition 4.1]{KP11}
 The $\imath$quantum group $^{\imath}\bfU(\mathcal{S})$ associated with above Satake diagrams is generated by elements $\{B_i\,|\,i\in I\}$ and $\{K_i\,|\,i\in I\setminus \{j\}, \rho_n(j)=j\}$ subject to the following relations 
\begin{align*}
 &K_iK_{\rho_n(i)}=1,\quad K_i B_j= q^{j\cdot\rho_n(i)-j\cdot i}B_j K_i,\quad \mbox{for all $i,j\in I$},\\
	&B_j B_i - B_i B_j = \delta_{\rho_n(i), j} \frac{K_i -K_{\rho_n(i)}}{q-q^{-1}},\quad \mbox{if $i\cdot j=0$,}  \\
	&B_i^2 B_j -(q+q^{-1}) B_i B_j B_i + B_j B_i^2 = \delta_{i,\rho_n(i)}  B_j \\
	&\qquad\ - \delta_{j,\rho_n(i)}(q+q^{-1})B_i (q^{-1}\varsigma_iK_i+ q^2\varsigma_{\rho_n(i)} K_{\rho_n(i)}), \quad  \mbox{if $i\cdot j=-1$,}
\end{align*}
where $\varsigma_i=\delta_{i,r}+q^{-1}\delta_{i,r+1}$ for $i\in I$.
\end{df}

\begin{prop}\label{prop:autoU}
	There exist two algebra anti-automorphisms $\Omega$ and $\Psi$ on $^{\imath}\bfU(\mathcal{S})$ defined by
	\begin{align*}
		&\Omega(B_i)=B_{\rho_n(i)},\quad \Omega(K_i)=K_{\rho_n(i)},\quad \Omega(q)=q^{-1},\\
		&\Psi(B_i)=B_i,\quad \Psi(K_i)=q^{-\delta_{i,r}\delta_{r\cdot\rho_n(r),-1}}K_{\rho_n(i)},\quad \Psi(q)=q.
	\end{align*}
\end{prop}

\begin{proof}
It is straightforward to verify by direct computation.
\end{proof}

We denote by $\mathbf U^\jmath$ (resp. $\mathbf U^\imath$) the $\imath$quantum group corresponding to Satake diagram \RomanNumeralCaps{1} (resp. \RomanNumeralCaps{2}). Let $e\in\{1,-1\}$. We set $[x, y]_{e}=xy-q^{e}yx$. The following Proposition is slight modification of coefficients in \cite[Theorem 4.3, Theorem 4.4, Theorem 4.6, Theorem 4.7]{KP11}. The proof is exactly the same as that in \cite[Section 4.5]{KP11}.

\begin{prop}[{\cite[Theorem 4.3--4.4, Theorem 4.6--4.7]{KP11}}]\label{prop:braidU}
There exist unique automorphisms $\tau'_{i,e}$ and $\tau''_{i,-e}$ $(i \in\{1,2,\cdots,r\})$ on $\bfU^\jmath$ satisfying the following:
\begin{align}
	\tau'_{i,e}(K_j) &= \tau''_{i,-e}(K_j) = \begin{cases}
		K_{\rho_n(i)}, \quad & \text{if } j = i\neq r, \\
	K_iK_j, \quad & \text{if } |i - j| = 1,\ i\neq r, \\
	K_j, \quad & \text{otherwise}.
	\end{cases} \nonumber 
\end{align}

If $1\leq i\leq r-1$,
	\begin{align}
	\tau'_{i,e}(B_j) &= \begin{cases}
			-B_{\rho_n(i)}K_{\rho_n(i)}^e, & \text{if } i = j, \\
			-K_i^{e}B_i,  & \text{if } i =\rho_n(j), \\
			[B_i, B_j]_{-e}, \quad & \text{if } i\cdot j=-1, \\
			[B_j, B_{\rho_n(i)}]_{e}, \quad & \text{if } i\cdot\rho_n(j)=-1, \\
		B_j, \quad & \text{otherwise},
		\end{cases}	\nonumber \\	
   \tau''_{i,-e}(B_j) &= \begin{cases}
		-K_{i}^eB_{\rho_n(i)}, & \text{if } i = j, \\
		-B_i K_{\rho_n(i)}^{e},  & \text{if } i =\rho_n(j), \\
		[B_j, B_i]_{-e}, \quad & \text{if } i\cdot j=-1, \\
		[B_{\rho_n(i)}, B_j]_{e}, \quad & \text{if } i\cdot\rho_n(j)=-1, \\
		B_j, \quad & \text{otherwise}.
	\end{cases} \nonumber
		\end{align}
	
If $i=r$,	
	\begin{align}	
		\tau'_{r,e}(B_j) &= \begin{cases}
	q^{-e}[[B_{r-1},B_r]_e,B_{r+1}]_e-K_r^eB_{r-1}, & \text{if } j= r-1, \\
K_r^{e}B_r,  & \text{if } j=r, \\
	B_{r+1}K_{\rho_n(r)}^e, \quad & \text{if } j=r+1, \\
q^{-e}[[B_{r+2},B_{r+1}]_e,B_{r}]_e-B_{r+2}K_{\rho_n(r)}^e, \quad & \text{if } j=r+2, \\
	B_j, \quad & \text{otherwise},
\end{cases}	\nonumber \\
		\tau''_{r,-e}(B_j) &= \begin{cases}
	q^{-e}[B_{r+1},[B_r,B_{r-1}]_e]_e-K_{\rho_n(r)}^eB_{r-1}, & \text{if } j= r-1, \\
K_{\rho_n(r)}^{e}B_r, & \text{if } j=r, \\
	B_{r+1}K_{r}^e, \quad & \text{if } j=r+1, \\
	q^{-e}[B_{r},[B_{r+1},B_{r+2}]_e]_e-B_{r+2}K_{r}^e, \quad & \text{if } j=r+2, \\
	B_j, \quad & \text{otherwise}.
\end{cases}	\nonumber 
	\end{align}

There exist unique automorphisms $\tau'_{i,e}$ and $\tau''_{i,-e}$ $(i \in\{1,2,\cdots,r+1\})$ on $\bfU^\imath$ satisfying the following:
\begin{align}
	\tau'_{i,e}(K_j) &= \tau''_{i,-e}(K_j) = \begin{cases}
	K_{\rho_n(i)}, \quad & \text{if } i = j, \\
	K_iK_j, \quad & \text{if } |i - j| = 1,\ i\neq r+1, \\
	K_j, \quad & \text{otherwise}.
	\end{cases} \nonumber 
\end{align}

If $1\leq i\leq r$,
	\begin{align}
	\tau'_{i,e}(B_j) &= \begin{cases}
		-B_{\rho_n(i)}K_{\rho_n(i)}^e, & \text{if } i = j, \\
		-K_i^{e}B_i,  & \text{if } i =\rho_n(j), \\
		[B_i, B_j]_{-e}, \quad & \text{if } i\cdot j=-1,\ \rho_n(i)\cdot j\neq -1, \\
		[B_j, B_{\rho_n(i)}]_{e}, \quad & \text{if }i\cdot j\neq -1,\ \rho_n(i)\cdot j=-1, \\
		[B_{r},[B_{r+1},B_{r+2}]_e]_{-e}+B_{r+1}K_{\rho_n(r)}^e, \quad & \text{if }i\cdot j= -1,\ \rho_n(i)\cdot j=-1, \\
		B_j, \quad & \text{otherwise},
	\end{cases}	\nonumber \\	
	\tau''_{i,-e}(B_j) &= \begin{cases}
	-K_{i}^eB_{\rho_n(i)}, & \text{if } i = j, \\
	-B_iK_{\rho_n(i)}^{e},  & \text{if } i =\rho_n(j), \\
	[B_j, B_i]_{-e}, \quad & \text{if } i\cdot j=-1,\ \rho_n(i)\cdot j\neq -1, \\
	[B_{\rho_n(i)}, B_j]_{e}, \quad & \text{if }i\cdot j\neq -1,\ \rho_n(i)\cdot j=-1, \\
	[[B_{r},B_{r+1}]_e,B_{r+2}]_{-e}+B_{r+1}K_{\rho_n(r)}^e, \quad & \text{if }i\cdot j= -1,\ \rho_n(i)\cdot j=-1, \\
	B_j, \quad & \text{otherwise}.
\end{cases}	 \nonumber
\end{align}

If $i=r+1$,	
\begin{align}	
	\tau'_{r+1,e}(B_j) &= \begin{cases}
		[B_{r+1},B_{r}]_{-e}, & \text{if } j= r, \\
	[B_{r+2},B_{r+1}]_e, \quad & \text{if } j=r+2, \\	
		B_j, \quad & \text{otherwise},
	\end{cases}	\nonumber \\
	\tau''_{r+1,-e}(B_j) &= \begin{cases}
		[B_{r+1},B_{r}]_{e}, & \text{if } j= r, \\
		[B_{r+2},B_{r+1}]_{-e}, \quad & \text{if } j=r+2, \\	
		B_j, \quad & \text{otherwise}.
	\end{cases}	\nonumber 
\end{align}

Let $\mathbb I=\{1,2,\cdots,[\frac{n+1}{2}]\}$.
The automorphisms  $\{ \tau'_{i,e} \}_{i \in\mathbb I}$ and $\{ \tau''_{i,-e} \}_{i \in\mathbb I}$ are inverse each other, i.e., $\tau'_{i,e}\tau''_{i,-e}=\tau''_{i,-e}\tau'_{i,e}=\mathrm{id}$.
	Moreover, $\{ \tau'_{i,e} \}_{i \in\mathbb I}$ and $\{ \tau''_{i,-e} \}_{i \in\mathbb I}$ satisfy  braid relations of type $B_r$ or $B_{r+1}$, i.e., the following relations hold:
	\begin{align*}
		&\tau_{i-1}\tau_i\tau_{i-1}=\tau_i\tau_{i-1}\tau_i, \quad \text{if } 2 \leq i \leq [\frac{n+1}{2}]-1,\\
		&\tau_{i-1}\tau_i\tau_{i-1}\tau_i=\tau_i\tau_{i-1}\tau_i\tau_{i-1}, \quad \text{if } i=[\frac{n+1}{2}],\\
		&\tau_i\tau_j=\tau_j\tau_j, \quad \text{if } |i-j|\neq 1,
	\end{align*}
	where $\tau_i:=\tau_{i,e}'$ or $\tau_{i,-e}''$.
\end{prop}

\begin{prop}
The automorphism $\tau_i$ commutes with $\Omega$, i.e., $\tau_i\circ\Omega=\Omega\circ\tau_i$.
\end{prop}

\begin{proof}
	We only prove  the formula $\tau'_{r,e}\circ \Omega(B_{r+2})=\Omega  \circ \tau'_{r,e}(B_{r+2})$ for $\bfU^\jmath$. By Proposition \ref{prop:autoU} and Proposition \ref{prop:braidU}, we have
	\begin{align*}
		&\tau'_{r,e}\circ \Omega(B_{r+2})=\tau'_{r,e}	(B_{r-1})\\
		=&q^{-e}[[B_{r-1},B_r]_e,B_{r+1}]_e-K_r^{e}B_{r-1}\\
		=&q^{-e}((B_{r-1}B_r-q^eB_rB_{r-1})B_{r+1}-q^{e}B_{r+1}(B_{r-1}B_r-q^eB_rB_{r-1})-K_r^{e}B_{r-1}\\
		=&q^{-e}B_{r-1}B_{r}B_{r+1}-B_{r}B_{r-1}B_{r+1}-B_{r+1}B_{r-1}B_{r}+q^eB_{r+1}B_{r}B_{r-1}-K_r^{e}B_{r-1}\\
		=&\Omega(q^{-e}(B_{r+2}B_{r+1}-q^{e}B_{r+1}B_{r+2})B_{r}-B_{r}(B_{r+2}B_{r+1}-q^{e}B_{r+1}B_{r+2})-B_{r+2}K_r^{-e})\\
			=&\Omega(q^{-e}[[B_{r+2},B_{r+1}]_e,B_r]_e-B_{r+2}K_r^{-e})\\
		=&\Omega\circ \tau'_{r,e} (B_{r+2}).
	\end{align*}
 The proof of other cases is similar. 
\end{proof}

\section{Braid group action on modified $q$-Weyl algebra}\label{sec:qWeyl}

\begin{df}[{\cite[Definition 3.1]{FGH}}]\label{modified q-Weyl algebra}
	The modified $q$-Weyl algebra $\mathbf{A}_q(\mathcal{S})$ associated with Satake diagram \RomanNumeralCaps{1} (resp. \RomanNumeralCaps{2}) is generated by $\mathfrak d_i$, $\mathfrak x_i$, $\mathfrak m_i$ ($i\in \{1,2,\cdots,r,r+1\}$) over $\mathbb{Q}(q)$  subject to the following relations:
	\begin{align}
		&\mathfrak m_i\mathfrak m_i^{-1}=\mathfrak m_i^{-1}\mathfrak m_i=1,\
		\mathfrak m_i\mathfrak m_j=\mathfrak m_j\mathfrak m_i,\label{mimj}\\
		&\mathfrak d_i\mathfrak m_j=\mathfrak m_j\mathfrak d_i,\
		\mathfrak x_i\mathfrak m_j=\mathfrak m_j\mathfrak x_i,\
		\mathfrak d_i\mathfrak x_j=\mathfrak x_j\mathfrak d_i,\ \ \text{if $i\neq j$},\\	
		&\mathfrak d_i\mathfrak d_j=\mathfrak d_j\mathfrak d_i,\ \mathfrak x_i\mathfrak x_j=\mathfrak x_j\mathfrak x_i,\\
		&\mathfrak d_i\mathfrak m_i=q^{1+\delta_{i,\rho_n(r)}\delta_{i,r+1}}\mathfrak m_i\mathfrak d_i,\
		\mathfrak x_i\mathfrak m_i=q^{-1-\delta_{i,\rho_n(r)}\delta_{i,r+1}}\mathfrak m_i\mathfrak x_i,\label{dimiximi}\\
		&\mathfrak{d}_i\mathfrak{x}_i=\frac{q^{1+\delta_{i,\rho_n(r)}\delta_{i,r+1}}\mathfrak m_i-q^{1+\delta_{i,\rho_n(r)}\delta_{i,r+1}}\mathfrak m_i^{-1}}{q-q^{-1}},\ \mathfrak{x}_i\mathfrak{d}_i=\frac{\mathfrak m_i-\mathfrak m_i^{-1}}{q-q^{-1}}.\label{dixixidi}
	\end{align}

\end{df}

\begin{prop}\label{prop:autoAq}
	There exist two anti-automorphisms $\omega$, $\psi$ on $\mathbf A_q(\mathcal{S})$ defined by
	\begin{align*}
		&\omega(\mathfrak x_i)=\mathfrak d_i,\quad \omega(\mathfrak d_i)=\mathfrak x_i,\quad \omega(\mathfrak m_i)=\mathfrak m_i^{-1},\quad \omega(q)=q^{-1},\\
		&\psi(\mathfrak x_i)=(-1)^{i+1}\mathfrak x_i,\ \psi(\mathfrak d_i)=(-1)^{i}\mathfrak d_i,\ \psi(\mathfrak m_i)=q^{\delta_{i,r+1}(i\cdot\rho_n(i)-2\delta_{i,\rho_n(i)})-1}\mathfrak m_i^{-1},\ \psi(q)=q.
	\end{align*}

\end{prop}

\begin{proof}
	The verifications of anti-automorphisms $\omega,~\psi$ are easy and will be skipped.
\end{proof}

\begin{prop}[{\cite[Theorem 4.1]{FGH}}]\label{prop:UjtoAq}
There exists a $\mathbb Q(q)$-algebra homomorphism $\varphi: {^{\imath}\bfU(\mathcal{S})} \rightarrow \mathbf{A}_q(\mathcal{S})$ which is given by
	\begin{align}
B_i &= \begin{cases}
		\mathfrak x_{i+1}\mathfrak d_i, & \text{if } 1\leq i\leq r, \\
	\mathfrak x_{\rho_n(i)}\mathfrak d_{1+\rho_n(i)},  & \text{if } \rho_n(r)\leq i\leq \rho_n(1), \\
	\mathfrak x_{r+1}\mathfrak d_{r+1}, & \text{if } i=r+1,\ \rho_n(r+1)=r+1,
	\end{cases}	\nonumber \\	
K_i &= \begin{cases}
	q^{-\delta_{i,r}}\mathfrak m_i\mathfrak m_{i+1}^{-1}, & \text{if } 1\leq i\leq r=\rho_n(r+1), \\
	q^{\delta_{i,r+1}}\mathfrak m_{\rho_n(i)}^{-1}\mathfrak m_{1+\rho_n(i)}, & \text{if } r+1=\rho_n(r)\leq i\leq \rho_n(1),\\
	\mathfrak m_r\mathfrak m_{r+1}^{-1}, & \text{if } i=r=\rho_n(r+2).\\
	\end{cases} \nonumber
\end{align}
\end{prop}

Let $\mathbf A_q^\jmath$ (resp. $\mathbf A_q^\imath$) denote the modified $q$-Weyl algebra corresponding to Satake diagram \RomanNumeralCaps{1} (resp. \RomanNumeralCaps{2}).

\begin{thm}\label{thm:braidAq}
There exist automorphisms $T'_{i,e}$ and $T''_{i,-e}$ $(i \in\{1,2,\cdots,r\})$ on $\mathbf A_q^\jmath$ satisfying the following:	
\begin{align*}
	&T'_{i,e}(\mathfrak m_j)=T''_{i,-e}(\mathfrak m_j)=\begin{cases}
		\mathfrak m_i,  & \text{if } i\neq r,\ j=i+1,\\
		\mathfrak m_{i+1},  & \text{if } i\neq r,\ j=i,\\
		\mathfrak m_j,  & \text{otherwise},\\
	\end{cases}	
	\\
	&T'_{i,e}(\mathfrak d_j)=\begin{cases}
		q^{e}\mathfrak m_r^{-2e}\mathfrak d_{r+1},  & \text{if } i=r,\ j=i+1,\\
		q^{-2e}\mathfrak m_r^{-e}\mathfrak m_{r+1}^{-e}\mathfrak d_r,  & \text{if } i=j=r,\\
		-q^{-e}\mathfrak m_{i+1}^{-e}\mathfrak d_i,  & \text{if } i\neq r,\ j=i+1,\\
		\mathfrak m_i^{-e}\mathfrak d_{i+1},  & \text{if } i=j\neq r,\\
		\mathfrak d_j,  & \text{otherwise},
	\end{cases}	
	\\
	&T'_{i,e}(\mathfrak x_j)=\begin{cases}
		q^{-e}\mathfrak m_r^{2e}\mathfrak x_{r+1},  & \text{if } i=r,\ j=i+1,\\
		q^{e}\mathfrak m_r^{e}\mathfrak m_{r+1}^{e}\mathfrak x_r,  & \text{if } i=j=r,\\
		-q^{e}\mathfrak m_{i+1}^{e}\mathfrak x_i,  & \text{if } i\neq r,\ j=i+1,\\
		\mathfrak m_i^{e}\mathfrak x_{i+1},  & \text{if } i=j\neq r,\\
		\mathfrak x_j,  & \text{otherwise},
	\end{cases}	
	\\
&T''_{i,-e}(\mathfrak d_j)=\begin{cases}
	q^{-e}\mathfrak m_r^{2e}\mathfrak d_{r+1},  & \text{if } i=r,\ j=i+1,\\
	q^{2e}\mathfrak m_r^{e}\mathfrak m_{r+1}^{e}\mathfrak d_r,  & \text{if } i=j=r,\\
	\mathfrak m_{i+1}^{e}\mathfrak d_i,  & \text{if } i\neq r,\ j=i+1,\\
	-q^{e}\mathfrak m_i^{e}\mathfrak d_{i+1},  & \text{if } i=j\neq r,\\
	\mathfrak d_j,  & \text{otherwise},
\end{cases}	
\\
&T''_{i,-e}(\mathfrak x_j)=\begin{cases}
	q^{e}\mathfrak m_r^{-2e}\mathfrak x_{r+1},  & \text{if } i=r,\ j=i+1,\\
	q^{-e}\mathfrak m_r^{-e}\mathfrak m_{r+1}^{-e}\mathfrak x_r,  & \text{if } i=j=r,\\
\mathfrak m_{i+1}^{-e}\mathfrak x_i,  & \text{if } i\neq r,\ j=i+1,\\
		-q^{-e}\mathfrak m_i^{-e}\mathfrak x_{i+1},  & \text{if } i=j\neq r,\\
	\mathfrak x_j,  & \text{otherwise}.
\end{cases}	
\end{align*}

There exist automorphisms $T'_{i,e}$ and $T''_{i,-e}$ $(i \in\{1,2,\cdots,r+1\})$ on $\mathbf A_q^\imath$ satisfying the following:	
\begin{align*}
	&T'_{i,e}(\mathfrak m_j)=T''_{i,-e}(\mathfrak m_j)=\begin{cases}
		\mathfrak m_i,  & \text{if } j=i+1,\\
		\mathfrak m_{i+1},  & \text{if } j=i\neq r+1,\\
		\mathfrak m_j,  & \text{otherwise},\\
	\end{cases}	
	\\
	&T'_{i,e}(\mathfrak d_j)=\begin{cases}
	\mathfrak m_{r+1}^{-e}\mathfrak d_{r+1},  & \text{if } i=j=r+1,\\
		-q^{-e}\mathfrak m_{i+1}^{-e}\mathfrak d_i,  & \text{if } j=i+1,\\
		\mathfrak m_i^{-e}\mathfrak d_{i+1},  & \text{if } i=j\neq r+1,\\
		\mathfrak d_j,  & \text{otherwise},
	\end{cases}	
	\\
	&T'_{i,e}(\mathfrak x_j)=\begin{cases}
		q^{-e}\mathfrak m_{r+1}^{e}\mathfrak x_{r+1},  & \text{if } i=j=r+1,\\
		-q^{e}\mathfrak m_{i+1}^{e}\mathfrak x_i,  & \text{if } j=i+1,\\
		\mathfrak m_i^{e}\mathfrak x_{i+1},  & \text{if } i=j\neq r+1,\\
		\mathfrak x_j,  & \text{otherwise},
	\end{cases}	
	\\
	&T''_{i,-e}(\mathfrak d_j)=\begin{cases}
		\mathfrak m_{r+1}^{e}\mathfrak d_{r+1},  & \text{if } i=j=r+1,\\
		\mathfrak m_{i+1}^{e}\mathfrak d_i,  & \text{if }  j=i+1,\\
		-q^e\mathfrak m_i^{e}\mathfrak d_{i+1},  & \text{if } i=j\neq r+1,\\
		\mathfrak d_j,  & \text{otherwise},
	\end{cases}	
	\\
	&T''_{i,-e}(\mathfrak x_j)=\begin{cases}
		q^e\mathfrak m_{r+1}^{-e}\mathfrak x_{r+1},  & \text{if } i=j=r+1,\\
		\mathfrak m_{i+1}^{-e}\mathfrak x_i,  & \text{if } j=i+1,\\
		-q^{-e}\mathfrak m_i^{-e}\mathfrak x_{i+1},  & \text{if } i=j\neq r+1,\\
		\mathfrak x_j,  & \text{otherwise}.
	\end{cases}	
\end{align*}

Moreover, we have $T'_{i,e}T''_{i,-e}=T''_{i,-e}T'_{i,e}=\mathrm{id}$.
\end{thm}
\begin{proof}
We need to show relations \eqref{mimj}--\eqref{dixixidi} hold under the action of operators $T'_{i,e}$ and $T''_{i,-e}$.	
We only consider $T_{r,e}'$ acting on  $\mathfrak {d}_{j_1}\mathfrak{x}_{j_2}$ $(j_1,j_2\in\{r,r+1\})$ for  $\mathbf A_q^\jmath$.
	
	If $j_1=j_2=r$, we have
	\begin{align*}
		T'_{r,e}(\mathfrak d_{r}\mathfrak x_{r})&=q^{-e}
		\mathfrak{m}_r^{-e}	\mathfrak{m}_{r+1}^{-e}
		\mathfrak{d}_{r}
		\mathfrak{m}_r^{e}\mathfrak{m}_{r+1}^{e}
		\mathfrak{x}_{r}=\mathfrak d_{r}\mathfrak x_{r}\\
		&=\frac{q\mathfrak m_{r}-q^{-1}\mathfrak m_{r}^{-1}}{q-q^{-1}}=
		T'_{r,e}(\frac{q\mathfrak m_{r}-q^{-1}\mathfrak m_{r}^{-1}}{q-q^{-1}}).
	\end{align*}
	
	If $j_1=j_2=r+1$, we have
	\begin{align*}
		T'_{r,e}(\mathfrak d_{r+1}\mathfrak x_{r+1})&=\mathfrak{m}_r^{-2e}\mathfrak{d}_{r+1}\mathfrak{m}_r^{2e}\mathfrak{x}_{r+1}=\mathfrak d_{r+1}\mathfrak x_{r+1}\\
		&=\frac{q^2\mathfrak m_{r+1}-q^{-2}\mathfrak m_{r+1}^{-1}}{q-q^{-1}}=
		T'_{r,e}(\frac{q^2\mathfrak m_{r+1}-q^{-2}\mathfrak m_{r+1}^{-1}}{q-q^{-1}}).
	\end{align*}
	
	If $j_1=r+1,~j_2=r$, we have
	\begin{align*}
		T'_{r,e}(\mathfrak d_{r+1}\mathfrak x_{r})&=q^{2e}\mathfrak m_r^{-2e}\mathfrak d_{r+1}\mathfrak m_r^{e}\mathfrak m_{r+1}^e\mathfrak x_r\\
		&=q^{2e}\mathfrak m_r^{e}\mathfrak m_{r+1}^{e}\mathfrak x_r\mathfrak m_r^{-2e}\mathfrak d_{r+1}
		=T'_{r,e}(\mathfrak x_{r}\mathfrak d_{r+1}).	
	\end{align*}

	If $j_1=r,~j_2=r+1$, we have
	\begin{align*}
		T'_{r,e}(\mathfrak d_{r}\mathfrak x_{r+1})&=q^{-3e}\mathfrak m_r^{-e}\mathfrak m_{r+1}^{-e}\mathfrak d_{r}\mathfrak m_r^{2e}\mathfrak x_{r+1}\\
		&=q^{-3e}\mathfrak m_r^{2e}\mathfrak x_{r+1}\mathfrak m_r^{-e}\mathfrak m_{r+1}^{-e}
		\mathfrak d_{r}=
		T'_{r,e}(\mathfrak x_{r+1}\mathfrak d_{r}).
	\end{align*}
 The verifications for the relations in $\mathbf A_q^\imath$ follows a similar approach.
\end{proof}

\begin{prop}
The automorphism $T_i:=T_{i,e}'$ or $T_{i,-e}''$ commutes with $\omega$, i.e., 
$
\omega\circ T'_{i}=T'_{i}\circ\omega
$.
\end{prop}

\begin{proof}
 For $\mathbf A_q^\jmath$, we give the proof only for the formula $\omega\circ T'_{r,e}(d_{j})=T'_{r,e} \circ \omega(d_{j})$ such that $j=r,r+1$. 
\vskip 2mm	
	If $j=r$, we have
	\begin{align*}
		\omega\circ T'_{r,e}(\mathfrak d_r)=\omega(q^{-2e}\mathfrak m_r^{-e}\mathfrak m_{r+1}^{-e}\mathfrak d_r)=&q^{2e}\mathfrak x_r\mathfrak m_{r}^{e}\mathfrak m_{r+1}^{e}\\
		=&q^{e}\mathfrak m_{r}^{e}\mathfrak m_{r+1}^{e}\mathfrak x_r=T'_{r,e}(\mathfrak x_r)=T'_{r,e}\circ\omega(\mathfrak d_r).
	\end{align*}
	
	If $j=r+1$, we have
	\begin{align*}
		\omega\circ T'_{r,e}(\mathfrak d_{r+1})=\omega(q^e\mathfrak m_r^{-2e}\mathfrak d_{r+1})=q^{-e}\mathfrak x_{r+1}\mathfrak m_r^{2e}=T'_{r,e}(\mathfrak x_{r+1})=T'_{r,e}\circ\omega(\mathfrak d_{r+1}).
	\end{align*}
The proof for $\mathbf A_q^\imath$ is similar.
	\end{proof}

\begin{thm}
The following braid relations hold.
	\begin{align}
	&T_{i-1}T_iT_{i-1}=T_iT_{i-1}T_i, \quad \text{if } 2 \leq i \leq [\frac{n+1}{2}]-1,\label{eq:TTT}\\
	&T_{i-1}T_iT_{i-1}T_i=T_iT_{i-1}T_iT_{i-1}, \quad \text{if } i=[\frac{n+1}{2}],\label{eq:TTTT}\\
	&T_iT_j=T_jT_j, \quad \text{if } |i-j|\neq 1,\label{eq:TT}
\end{align}
where $T_i:=T_{i,e}'$ or $T_{i,-e}''$.
\end{thm}

\begin{proof}
For $\mathbf A_q^{\jmath}$, we only show the proof for the relation in \eqref{eq:TTTT}, the proof for other relations is similar. It is enough to consider the case of $T_i=T'_{i,e}$ due to $T'_{i,e}T''_{i,-e}=T''_{i,-e}T'_{i,e}=\mathrm{id}$. We only need to show the following relations hold for $j\in \{r, r+1, r-1\}$.
\begin{align}
&T_{r-1}T_rT_{r-1}T_r(\mathfrak d_j)=T_rT_{r-1}T_rT_{r-1}(\mathfrak d_j),\label{Tdj}\\
&T_{r-1}T_rT_{r-1}T_r(\mathfrak x_j)=T_rT_{r-1}T_rT_{r-1}(\mathfrak x_j),\label{Txj}\\
&T_{r-1}T_rT_{r-1}T_r(\mathfrak m_j)=T_rT_{r-1}T_rT_{r-1}(\mathfrak m_j).\label{Tmj}
\end{align}

The verification of \eqref{Tmj} is easy, we omit it here.
\vskip 2mm
For \eqref{Tdj}, if $j=r$, we have
\begin{align*}
&	T_{r-1}T_rT_{r-1}T_r(\mathfrak d_r)-T_rT_{r-1}T_rT_{r-1}(\mathfrak d_r)	\\
=&T_{r-1}T_rT_{r-1}(q^{-2e}\mathfrak m_r^{-e}\mathfrak m_{r+1}^{-e}\mathfrak d_r)-T_rT_{r-1}T_{r}(-q^{-e}\mathfrak m_r^{-e}\mathfrak d_{r-1})\\
=&T_{r-1}T_r(-q^{-3e}\mathfrak m_{r-1}^{-e}\mathfrak m_{r+1}^{-e}\mathfrak m_{r}^{-e}\mathfrak d_{r-1})-T_rT_{r-1}(-q^{-e}\mathfrak m_r^{-e}\mathfrak d_{r-1})\\
=&T_{r-1}(-q^{-3e}\mathfrak m_{r-1}^{-e}\mathfrak m_{r+1}^{-e}\mathfrak m_{r}^{-e}\mathfrak d_{r-1})-T_r(-q^{-e}\mathfrak m_{r-1}^{-2e}\mathfrak d_{r})\\
=&-q^{-3e}\mathfrak m_{r}^{-e}\mathfrak m_{r-1}^{-2e}\mathfrak m_{r+1}^{-e}\mathfrak d_{r}+q^{-3e}\mathfrak m_{r}^{-e}\mathfrak m_{r-1}^{-2e}\mathfrak m_{r+1}^{-e}\mathfrak d_{r}=0.
\end{align*}

If $j=r+1$, we have
\begin{align*}
&	T_{r-1}T_rT_{r-1}T_r(\mathfrak d_{r+1})	-T_rT_{r-1}T_rT_{r-1}(\mathfrak d_{r+1})	\\
=&T_{r-1}T_rT_{r-1}(q^{e}\mathfrak m_r^{-2e}\mathfrak d_{r+1})-T_rT_{r-1}T_{r}(\mathfrak d_{r+1})\\
=&T_{r-1}T_r(q^{e}\mathfrak m_{r-1}^{-2e}\mathfrak d_{r+1})-T_rT_{r-1}(q^{e}\mathfrak m_r^{-2e}\mathfrak d_{r+1})\\
=&T_{r-1}(q^{2e}\mathfrak m_{r-1}^{-2e}\mathfrak m_{r}^{-2e}\mathfrak d_{r+1})-T_r(q^{e}\mathfrak m_{r-1}^{-2e}\mathfrak d_{r+1})\\
=&q^{2e}\mathfrak m_{r}^{-2e}\mathfrak m_{r-1}^{-2e}\mathfrak d_{r+1}-q^{2e}\mathfrak m_{r}^{-2e}\mathfrak m_{r-1}^{-2e}\mathfrak d_{r+1}=0.
\end{align*}

If $j=r-1$, we have
\begin{align*}
&	T_{r-1}T_rT_{r-1}T_r(\mathfrak d_{r-1})	-T_rT_{r-1}T_rT_{r-1}(\mathfrak d_{r-1})	\\
=&T_{r-1}T_rT_{r-1}(\mathfrak d_{r-1})-T_rT_{r-1}T_{r}(\mathfrak m_{r-1}^{-e}\mathfrak d_{r})\\
=&T_{r-1}T_r(\mathfrak m_{r-1}^{-e}\mathfrak d_{r})-T_rT_{r-1}(q^{-2e}\mathfrak m_{r-1}^{-e}\mathfrak m_{r}^{-e}\mathfrak m_{r+1}^{-e}\mathfrak d_{r})\\
=&T_{r-1}(q^{-2e}\mathfrak m_{r-1}^{-e}\mathfrak m_{r}^{-e}\mathfrak m_{r+1}^{-e}\mathfrak d_{r})-T_r(-q^{-3e}\mathfrak m_{r}^{-2e}\mathfrak m_{r-1}^{-e}\mathfrak m_{r+1}^{-e}\mathfrak d_{r-1})\\
=&-q^{-3e}\mathfrak m_{r}^{-2e}\mathfrak m_{r-1}^{-e}\mathfrak m_{r+1}^{-e}\mathfrak d_{r-1}+q^{-3e}\mathfrak m_{r}^{-2e}\mathfrak m_{r-1}^{-e}\mathfrak m_{r+1}^{-e}\mathfrak d_{r-1}=0.
\end{align*}

Hence  \eqref{Tdj} holds for $j=r, r+1, r-1$.
\vskip 2mm
For \eqref{Txj}, if $j=r$, we have
\begin{align*}
&	T_{r-1}T_rT_{r-1}T_r(\mathfrak x_r)	-	T_rT_{r-1}T_rT_{r-1}(\mathfrak x_r)	\\
=&T_{r-1}T_rT_{r-1}(q^{e}\mathfrak m_r^{e}\mathfrak m_{r+1}^{e}\mathfrak x_r)-T_rT_{r-1}T_{r}(-q^{e}\mathfrak m_r^{e}\mathfrak x_{r-1})\\
=&T_{r-1}T_r(-q^{2e}\mathfrak m_{r}^{e}\mathfrak m_{r-1}^{e}\mathfrak m_{r+1}^{e}\mathfrak x_{r-1})-T_rT_{r-1}(-q^{e}\mathfrak m_r^{e}\mathfrak x_{r-1})\\
=&T_{r-1}(-q^{2e}\mathfrak m_{r}^{e}\mathfrak m_{r-1}^{e}\mathfrak m_{r+1}^{e}\mathfrak x_{r-1})-T_r(-q^{e}\mathfrak m_{r-1}^{2e}\mathfrak x_{r})\\
=&-q^{2e}\mathfrak m_{r-1}^{2e}\mathfrak m_{r}^{e}\mathfrak m_{r+1}^{e}\mathfrak x_{r}+q^{2e}\mathfrak m_{r-1}^{2e}\mathfrak m_{r}^{e}\mathfrak m_{r+1}^{e}\mathfrak x_{r}=0.
\end{align*}

If $j=r+1$, we have

\begin{align*}
	&T_{r-1}T_rT_{r-1}T_r(\mathfrak x_{r+1})-		T_rT_{r-1}T_rT_{r-1}(\mathfrak x_{r+1})\\
	=&T_{r-1}T_rT_{r-1}(q^{-e}\mathfrak m_r^{2e}\mathfrak x_{r+1})-T_rT_{r-1}T_{r}(\mathfrak x_{r+1})\\
	=&T_{r-1}T_r(q^{-e}\mathfrak m_{r-1}^{2e}\mathfrak x_{r+1})-T_rT_{r-1}(q^{-e}\mathfrak m_r^{2e}\mathfrak x_{r+1})\\
	=&T_{r-1}(q^{-2e}\mathfrak m_{r-1}^{2e}\mathfrak m_{r}^{2e}\mathfrak x_{r+1})-T_r(q^{-e}\mathfrak m_{r-1}^{2e}\mathfrak x_{r+1})\\
	=&q^{-2e}\mathfrak m_{r}^{2e}\mathfrak m_{r-1}^{2e}\mathfrak x_{r+1}-q^{-2e}\mathfrak m_{r}^{2e}\mathfrak m_{r-1}^{2e}\mathfrak x_{r+1}=0.
\end{align*}

If $j=r-1$, we have
\begin{align*}
	&T_{r-1}T_rT_{r-1}T_r(\mathfrak x_{r-1})-	T_rT_{r-1}T_rT_{r-1}(\mathfrak x_{r-1})\\
	=&T_{r-1}T_rT_{r-1}(\mathfrak x_{r-1})-	T_rT_{r-1}T_{r}(\mathfrak m_{r-1}^{e}\mathfrak x_{r})\\
	=&T_{r-1}T_r(\mathfrak m_{r-1}^{e}\mathfrak x_{r})-	
	T_rT_{r-1}(q^{e}\mathfrak m_{r-1}^{e}\mathfrak m_{r}^{e}\mathfrak m_{r+1}^{e}\mathfrak x_{r})\\
	=&T_{r-1}(q^{e}\mathfrak m_{r-1}^{e}\mathfrak m_{r}^{e}\mathfrak m_{r+1}^{e}\mathfrak x_{r})-
	T_r(-q^{2e}\mathfrak m_{r}^{2e}\mathfrak m_{r-1}^{e}\mathfrak m_{r+1}^{e}\mathfrak x_{r-1})\\
	=&-q^{2e}\mathfrak m_{r}^{2e}\mathfrak m_{r-1}^{e}\mathfrak m_{r+1}^{e}\mathfrak x_{r-1}+
	q^{2e}\mathfrak m_{r}^{2e}\mathfrak m_{r-1}^{e}\mathfrak m_{r+1}^{e}\mathfrak x_{r-1}=0.
\end{align*}

Hence  \eqref{Txj} holds for $j\in \{r, r+1, r-1\}$.  We can similarly verify the braid relations \eqref{eq:TTT}--\eqref{eq:TT}  for $\mathbf A_q^{\imath}$, which gives the desired result.
\end{proof}

\begin{thm}\label{thm:Tphi=phitau}
We have the following intertwining relations
\begin{equation}
\omega\circ\varphi=\varphi\circ\Omega,\quad 
T_i\circ\varphi=\varphi\circ\tau_i.
\end{equation}
\end{thm}

\begin{proof}
For $\bfU^\jmath$ and $\mathbf A_q^\jmath$,	we only prove the formula $T'_{i,e}\circ \varphi(B_{j})=\varphi\circ \tau'_{i,e}(B_{j})$, where $1\leq i,j\leq r$. The following calculations stem from Proposition \ref{prop:braidU}, Proposition \ref{prop:UjtoAq} and Theorem \ref{thm:braidAq}.
	
	If $i=j=r$, we have
	\begin{align*}
		T'_{r,e}\circ\varphi(B_{r})=T'_{r,e}(\mathfrak x_{r+1}\mathfrak d_{r})
		=q^{-e}\mathfrak m_r^{e}\mathfrak m_{r+1}^{-e}
			\mathfrak x_{r+1}\mathfrak d_{r}=\varphi(K_r^{e}B_{r})=\varphi\circ\tau'_{r,e}(B_{r}).
	\end{align*}
	
	If $i=r,~j=r-1$, we have
	\begin{align*}
		T'_{r,e}\circ \varphi(B_{r-1})=T'_{r,e}(\mathfrak x_{r}\mathfrak d_{r-1})=q^{e}\mathfrak m_r^{e}\mathfrak m_{r+1}^{e}\mathfrak x_{r}\mathfrak d_{r-1},
	\end{align*}
	and the right-hand side,
	\begin{align*}
		&\varphi\circ\tau'_{r,e}(B_{r-1})
		=\varphi(q^{-e}[[B_{r-1},B_{r}]_e,B_{r+1}]_e-K_r^{e}B_{r-1})\\
		=&\varphi(q^{-e}B_{r-1}B_{r}B_{r+1}-B_{r}B_{r-1}B_{r+1}-B_{r+1}B_{r-1}B_{r}+q^eB_{r+1}B_{r}B_{r-1}-K_r^{e}B_{r-1})\\
		=&q^{-e}\mathfrak x_{r}\mathfrak d_{r-1}\mathfrak x_{r+1}\mathfrak d_{r}\mathfrak x_{r}\mathfrak d_{r+1}-\mathfrak x_{r+1}\mathfrak d_{r}\mathfrak x_{r}\mathfrak d_{r-1}\mathfrak x_{r}\mathfrak d_{r+1}-\mathfrak x_{r}\mathfrak d_{r+1}\mathfrak x_{r}\mathfrak d_{r-1}\mathfrak x_{r+1}\mathfrak d_{r}\\
		&+q^e\mathfrak x_{r}\mathfrak d_{r+1}\mathfrak x_{r+1}\mathfrak d_{r}\mathfrak x_{r}\mathfrak d_{r-1}
		-q^{-e}\mathfrak m_{r}^{e}\mathfrak m_{r+1}^{-e}\mathfrak x_{r}\mathfrak d_{r-1}\\
		=&(q-q^{-1})^{-2}\big((q^{-e}(\mathfrak m_r-\mathfrak m_r^{-1})-(q\mathfrak m_r-q^{-1}\mathfrak m_r^{-1}))(\mathfrak m_{r+1}-\mathfrak m_{r+1}^{-1}) \\
		&-((q^{-1}\mathfrak m_r-q\mathfrak m_r^{-1})-q^e(\mathfrak m_r-\mathfrak m_r^{-1}))(q^2\mathfrak m_{r+1}-q^{-2}\mathfrak m_{r+1}^{-1})\big)\mathfrak x_r\mathfrak d_{r-1}-q^{-e}\mathfrak m_r^{e}\mathfrak m_{r+1}^{-e}\mathfrak x_r\mathfrak d_{r-1}\\
		=&\begin{cases}
		q\mathfrak m_r\mathfrak m_{r+1}\mathfrak x_r\mathfrak d_{r-1}, & \text{if}\ e=1, \\
		q^{-1}\mathfrak m_r^{-1}\mathfrak m_{r+1}^{-1}\mathfrak x_r\mathfrak d_{r-1}, & \text{if}\ e=-1.
		\end{cases}
	\end{align*}

Hence $	T'_{r,e}\circ \varphi(B_{r-1})=\varphi\circ\tau'_{r,e}(B_{r-1})$.
\vskip 2mm	
	If $i\ne r,~j=i+1$, we have
	\begin{align*}
		T'_{i,e}\circ\varphi(B_{i+1})=T'_{i,e}(\mathfrak x_{i+2}\mathfrak d_{i+1})=-q^{-e}\mathfrak m_{i+1}^{-e}\mathfrak d_i\mathfrak x_{i+2},
	\end{align*}
	and the right-hand side,
	\begin{align*}
		\varphi\circ\tau'_{i,e}(B_{i+1})
		=&\varphi([B_{i},B_{i+1}]_{-e})
		=\mathfrak x_{i+1}\mathfrak d_{i}\mathfrak x_{i+2}\mathfrak d_{i+1}-q^{-e}\mathfrak x_{i+2}\mathfrak d_{i+1}\mathfrak x_{i+1}\mathfrak d_{i}\\
		=&(q-q^{-1})^{-1}\mathfrak d_i\mathfrak x_{i+2}((1-q^{-e+1})\mathfrak m_{i+1}+(q^{-e-1}-1)\mathfrak m_{i+1}^{-1})\\
		=&	\begin{cases}
		-q^{-1}\mathfrak d_i\mathfrak x_{i+2}\mathfrak m_{i+1}^{-1},& \text{if}\ e=1, \\
	    -q\mathfrak d_i\mathfrak x_{i+2}\mathfrak m_{i+1}, & \text{if}\ e=-1.
		\end{cases}
	\end{align*}

Hence $T'_{i,e}\circ\varphi(B_{i+1})=\varphi\circ\tau'_{i,e}(B_{i+1})$.
\vskip 2mm	
	If $i=j\ne r$, we have
	\begin{align*}
		T'_{i,e}\circ\varphi(B_{i})=T'_{i,e}(\mathfrak x_{i+1}\mathfrak d_{i})=	-\mathfrak x_{i}\mathfrak d_{i+1} \mathfrak m_{i+1}^{e}
			\mathfrak m_{i}^{-e}
=\varphi\circ\tau'_{i,e}(B_{i}).
	\end{align*}
	
	If $i\ne r$ and $j=i-1$, we have
	\begin{align*}
		T'_{i,e}\circ\varphi(B_{i-1})=T'_{i,e}(\mathfrak x_{i}\mathfrak d_{i-1})=\mathfrak m_i^{e}\mathfrak x_{i+1}\mathfrak d_{i-1},
	\end{align*}
	and the right-hand side,
	\begin{align*}
		\varphi\circ\tau'_{i,e}(B_{i-1})&=\varphi([B_{i},B_{i-1}]_{-e})
		=\mathfrak x_{i+1}\mathfrak d_i\mathfrak x_i\mathfrak d_{i-1}-q^{-e}\mathfrak x_i\mathfrak d_{i-1}\mathfrak x_{i+1}\mathfrak d_i\\
		&=(q-q^{-1})^{-1}((q-q^{-e})\mathfrak m_i+(q^{-e}-q^{-1})\mathfrak m_i^{-1})\mathfrak x_{i+1}\mathfrak d_{i-1}\\
		&=\begin{cases}
		\mathfrak m_i\mathfrak x_{i+1}\mathfrak d_{i-1},& \text{if}\ e=1, \\
		\mathfrak m_i^{-1}\mathfrak x_{i+1}\mathfrak d_{i-1}, & \text{if}\ e=-1.
		\end{cases}
	\end{align*}

Hence $T'_{i,e}\circ\varphi(B_{i-1})=\varphi\circ\tau'_{i,e}(B_{i-1})$.
\vskip 2mm	
	The proof of other cases is similar and we will skip it.
\end{proof}

\begin{thm}
The automorphism $T_i$ is unique such that the intertwining relation $T_i\circ\varphi=\varphi\circ\tau_i$ and braid relations \eqref{eq:TTT}--\eqref{eq:TT} hold.
\end{thm}

\begin{proof}
	We consider $\mathbf A_q^\jmath$, while skipping the similar proof for $\mathbf A_q^\imath$.	
For $1\leq i\leq r$, we denote by
\begin{equation*}
f_i=B_i,\quad e_i=B_{\rho_n(i)},\quad k_i=K_i, \quad k_i^{-1}=K_{\rho_n(i)}.
\end{equation*}
By Proposition \ref{prop:braidU}, we have 
	\begin{align}
	\tau'_{i,e}(f_j) &= \begin{cases}
		-e_ik_i^{-e}  & \text{if } j = i\neq r, \\
		k_i^{e}f_i  & \text{if } j = i= r, \\
		[f_i, f_j]_{-e} \quad & \text{if } |i - j| = 1,i\neq r \\
		q^{-e}[[f_j, f_{j+1}]_e,e_{j+1}]_e-k_{j+1}^ef_{j} \quad & \text{if } i=j+1=r,\\
		f_j \quad & \text{otherwise},
	\end{cases}\label{tauf}
\\
	\tau'_{i,e}(e_j) &= \begin{cases}
	-k_i^e f_i & \text{if } j = i\neq r, \\
	e_ik_i^{-e}  & \text{if } j = i= r, \\
	[e_j, e_i]_e \quad & \text{if } |i - j| = 1,i\neq r \\
	q^{-e}[[e_j, e_{j+1}]_e,f_{j+1}]_e-e_{j}k_{j+1}^{-e} \quad & \text{if } i=j+1=r,\\
	e_j \quad & \text{otherwise},
\end{cases} \label{taue}
\\
	\tau'_{i,e}(k_j) &= \begin{cases}
		k_i^{-1} \quad & \text{if } j = i\neq r, \\
		k_ik_j \quad & \text{if } |i - j| = 1, i\neq r \\
		k_j \quad & \text{otherwise}.
	\end{cases} \label{tauk}
\end{align}

We assume that there exists another $\mathbf T'_{i,e}:\mathbf{A}_q^\jmath\to \mathbf{A}_q^\jmath$ satisfying $\mathbf T'_{i,e}\circ\varphi=\varphi\circ\tau'_{i,e}$ and the braid relations \eqref{eq:TTT}--\eqref{eq:TT}. We first consider \eqref{tauf}, if $i=j+1=r$, we have $\mathbf T'_{r,e}\circ\varphi(f_{r-1})=\varphi\circ\tau'_{r,e}(f_{r-1})$. It follows that
\begin{align*}
\mathbf T'_{r,e}(\mathfrak x_{r}\mathfrak d_{r-1})
=&\varphi(q^{-e}[[f_{r-1}, f_{r}]_e,e_{r}]_e-k_{r}^ef_{r-1})\\
=&\varphi(q^{-e}f_{r-1}f_re_r-f_rf_{r-1}e_r-e_rf_{r-1}f_r+q^ee_rf_rf_{r-1}-k_r^ef_{r-1})\\
=&q^{-e}\mathfrak x_r\mathfrak d_{r-1}\mathfrak x_{r+1}\mathfrak d_r \mathfrak x_r\mathfrak d_{r+1}-\mathfrak x_{r+1}\mathfrak d_{r}\mathfrak x_{r}\mathfrak d_{r-1} \mathfrak x_r\mathfrak d_{r+1}-\mathfrak x_{r}\mathfrak d_{r+1}\mathfrak x_{r}\mathfrak d_{r-1} \mathfrak x_{r+1}\mathfrak d_{r}\\
&+q^e\mathfrak x_{r}\mathfrak d_{r+1}\mathfrak x_{r+1}\mathfrak d_{r} \mathfrak x_{r}\mathfrak d_{r-1}-{(q^{-1}\mathfrak m_r\mathfrak m_{r+1}^{-1})}^e\mathfrak x_{r}\mathfrak d_{r-1}\\
=&(q^{-e}\frac{\mathfrak m_r-\mathfrak m_r^{-1}}{q-q^{-1}}\frac{\mathfrak m_{r+1}-\mathfrak m_{r+1}^{-1}}{q-q^{-1}}-\frac{q\mathfrak m_r-q^{-1}\mathfrak m_r^{-1}}{q-q^{-1}}\frac{\mathfrak m_{r+1}-\mathfrak m_{r+1}^{-1}}{q-q^{-1}}\\
&-\frac{q^2\mathfrak m_{r+1}-q^{-2}\mathfrak m_{r+1}^{-1}}{q-q^{-1}}\frac{q^{-1}\mathfrak m_{r}-q\mathfrak m_{r}^{-1}}{q-q^{-1}}\\
&+
q^e\frac{\mathfrak m_{r}-\mathfrak m_{r}^{-1}}{q-q^{-1}}\frac{q^2\mathfrak m_{r+1}-q^{-2}\mathfrak m_{r+1}^{-1}}{q-q^{-1}}-{(q^{-1}\mathfrak m_r\mathfrak m_{r+1}^{-1})}^e)\mathfrak x_{r}\mathfrak d_{r-1}\\
=&
\begin{cases}
	q^{-1}\mathfrak m_r^{-1}\mathfrak m_{r+1}^{-1}\mathfrak x_r \mathfrak d_{r-1}, \quad & \text{if } e=-1 \\
		q \mathfrak m_r \mathfrak m_{r+1} \mathfrak x_r \mathfrak d_{r-1}, \quad & \text{if } e=1
	\end{cases}\\
=&q^{e}\mathfrak m_r^{e}\mathfrak m_{r+1}^{e}\mathfrak x_r \mathfrak d_{r-1}.
\end{align*}

If $i=j+1$, $i\neq r$,  we have
\begin{align*}
	\mathbf T'_{i,e}(\mathfrak x_{i}\mathfrak d_{i-1})
	=&\varphi(f_if_{i-1}-q^{-e}f_{i-1}f_i)\\
	=&\mathfrak x_{i+1}\mathfrak d_{i}\mathfrak x_i\mathfrak d_{i-1}-q^{-e}\mathfrak x_i\mathfrak d_{i-1}\mathfrak x_{i+1}\mathfrak d_{i}\\
	=&\frac{(q-q^{-e})\mathfrak m_i+(q^{-e}-q^{-1})\mathfrak m_i^{-1}}{q-q^{-1}}\mathfrak x_{i+1}\mathfrak d_{i-1}\\
	=&
	\begin{cases}
	\mathfrak m_i^{-1}\mathfrak x_{i+1}\mathfrak d_{i-1}, \quad & \text{if } e=-1 \\
	\mathfrak m_i\mathfrak x_{i+1}\mathfrak d_{i-1}, \quad & \text{if } e=1
	\end{cases}\\
	=&\mathfrak m_i^{e}\mathfrak x_{i+1}\mathfrak d_{i-1}.
\end{align*}

If $j=i+1$, $i\neq r$,  we have
\begin{align*}
	\mathbf T'_{i,e}(\mathfrak x_{i+2}\mathfrak d_{i+1})
	=&\varphi(f_if_{i+1}-q^{-e}f_{i+1}f_i)\\
	=&\mathfrak x_{i+1}\mathfrak d_{i}\mathfrak x_{i+2}\mathfrak d_{i+1}-q^{-e}\mathfrak x_{i+2}\mathfrak d_{i+1}\mathfrak x_{i+1}\mathfrak d_{i}\\
	=&\frac{(1-q^{1-e})\mathfrak m_{i+1}+(q^{-1-e}-1)\mathfrak m_{i+1}^{-1}}{q-q^{-1}}\mathfrak d_{i}\mathfrak x_{i+2}\\
	=&
	\begin{cases}
	-q	\mathfrak m_{i+1}\mathfrak d_{i}\mathfrak x_{i+2}, \quad & \text{if } e=-1 \\
	-q^{-1}	\mathfrak m_{i+1}^{-1}\mathfrak d_{i}\mathfrak x_{i+2}, \quad & \text{if } e=1
	\end{cases}\\
	=&-q^{-e}\mathfrak m_{i+1}^{-e}\mathfrak d_{i}\mathfrak x_{i+2}.
\end{align*}

It follows from the above calculations and \eqref{tauf} that
	\begin{align}\label{eq:TFxd}
	\mathbf T'_{i,e}(\mathfrak x_{j+1}\mathfrak d_{j}) &= \begin{cases}
		-\mathfrak x_i\mathfrak d_{i+1}{(\mathfrak m_i\mathfrak m_{i+1}^{-1})}^{-e}  & \text{if } j = i\neq r, \\
		{(q^{-1}\mathfrak m_i\mathfrak m_{i+1}^{-1})}^{e}\mathfrak x_{i+1}\mathfrak d_i  & \text{if } j = i= r, \\
		\mathfrak m_i^e\mathfrak x_{i+1}\mathfrak d_{i-1} \quad & \text{if } i=j+1,i\neq r \\
	-q^{-e}	\mathfrak m_{i+1}^{-e}\mathfrak d_{i}\mathfrak x_{i+2} \quad & \text{if } j=i+1,i\neq r \\
		q^e\mathfrak m_r^e\mathfrak m_{r+1}^e\mathfrak x_r \mathfrak d_{r-1} \quad & \text{if } i=j+1=r,\\
		\mathfrak x_{i+1}\mathfrak d_j \quad & \text{otherwise},
	\end{cases}
\end{align}

Using similar calculations on \eqref{taue} and \eqref{tauk}, we obtain
	\begin{align}\label{eq:TExd}
	\mathbf T'_{i,e}(\mathfrak x_{j}\mathfrak d_{j+1}) &= \begin{cases}
		-{(\mathfrak m_i\mathfrak m_{i+1}^{-1})}^{e}\mathfrak x_{i+1}\mathfrak d_i  & \text{if } j = i\neq r, \\
		\mathfrak x_i\mathfrak d_{i+1} {(q^{-1}\mathfrak m_i\mathfrak m_{i+1}^{-1})}^{-e} & \text{if } j = i= r, \\
		\mathfrak m_i^{-e}\mathfrak x_{i-1}\mathfrak d_{i+1} \quad & \text{if } i=j+1,i\neq r \\
		-q^{e}	\mathfrak m_{i+1}^{e}\mathfrak x_{i}\mathfrak d_{i+2} \quad & \text{if } j=i+1,i\neq r \\
		q^{-2e}\mathfrak m_r^{-e}\mathfrak m_{r+1}^{-e}\mathfrak x_{r-1} \mathfrak d_r \quad & \text{if } i=j+1=r,\\
		\mathfrak x_j\mathfrak d_{i+1} \quad & \text{otherwise},
	\end{cases}
\end{align}
and
	\begin{align}\label{eq:TKxd}
	\mathbf T'_{i,e}(q^{-\delta_{j,r}}\mathfrak m_j\mathfrak m_{j+1}^{-1}) &= \begin{cases}
		{(\mathfrak m_i\mathfrak m_{i+1}^{-1})}^{-1}\mathfrak x_{i+1}\mathfrak d_i  & \text{if } j = i\neq r, \\
		q^{-\delta_{j,r}}\mathfrak m_{r-1}\mathfrak m_r^{-1} \mathfrak m_j\mathfrak m_{j+1}^{-1} & \text{if } j =i+1=r,\\
		\mathfrak m_i\mathfrak m_{i+1}^{-1}\mathfrak m_{j}\mathfrak m_{j+1}^{-1} \quad & \text{if } |i-j|=1,i\neq r, j\neq r, \\
		q^{-\delta_{j,r}}\mathfrak m_j\mathfrak m_{j+1}^{-1} \quad & \text{otherwise}.
	\end{cases}
\end{align}

For $i=r$, by \eqref{eq:TFxd} and \eqref{eq:TExd}, we have the following system of equations
\begin{align}
&(\mathbf T'_{r,e}\mathfrak x_{r+1})(\mathbf T'_{r,e}\mathfrak d_r)={(q^{-1}\mathfrak m_i\mathfrak m_{i+1}^{-1})}^{e}\mathfrak x_{r+1}\mathfrak d_r,\label{eq:Txr+1Tdr}\\
&(\mathbf T'_{r,e}\mathfrak x_{r})(\mathbf T'_{r,e}\mathfrak d_{r-1})=q^e\mathfrak m_r^e\mathfrak m_{r+1}^e\mathfrak x_r \mathfrak d_{r-1},\\
&(\mathbf T'_{r,e}\mathfrak x_{r})(\mathbf T'_{r,e}\mathfrak d_{r+1})=\mathfrak x_r\mathfrak d_{r+1} {(q^{-1}\mathfrak m_r\mathfrak m_{r+1}^{-1})}^{-e},\\
&(\mathbf T'_{r,e}\mathfrak x_{r-1})(\mathbf T'_{r,e}\mathfrak d_{r})=q^{-2e}\mathfrak m_r^{-e}\mathfrak m_{r+1}^{-e}\mathfrak x_{r-1} \mathfrak d_r,\\
&(\mathbf T'_{r,e}\mathfrak x_{r-1})(\mathbf T'_{r,e}\mathfrak d_{r-2})=\mathfrak x_{r-1}\mathfrak d_{r-2},\\
&(\mathbf T'_{r,e}\mathfrak x_{r-2})(\mathbf T'_{r,e}\mathfrak d_{r-1})=\mathfrak x_{r-2}\mathfrak d_{r-1}.\label{eq:Txr-2Tdr-1}
\end{align}

For $i\neq r$,  by \eqref{eq:TFxd} and \eqref{eq:TExd} again, we have the following system of equations
\begin{align}
	&(\mathbf T'_{i,e}\mathfrak x_{i+1})(\mathbf T'_{i,e}\mathfrak d_i)=-\mathfrak x_i\mathfrak d_{i+1}{(\mathfrak m_i\mathfrak m_{i+1}^{-1})}^{-e},\label{eq:Txi+1Tdi}\\
	&(\mathbf T'_{i,e}\mathfrak x_{i})(\mathbf T'_{i,e}\mathfrak d_{i+1})=-{(\mathfrak m_i\mathfrak m_{i+1}^{-1})}^{e}\mathfrak x_{i+1}\mathfrak d_i,\\
    &(\mathbf T'_{i,e}\mathfrak x_{i})(\mathbf T'_{i,e}\mathfrak d_{i-1})=\mathfrak m_i^e\mathfrak x_{i+1}\mathfrak d_{i-1},\\
    &(\mathbf T'_{i,e}\mathfrak x_{i-1})(\mathbf T'_{i,e}\mathfrak d_{i})=\mathfrak m_i^{-e}\mathfrak x_{i-1}\mathfrak d_{i+1},\\
    &(\mathbf T'_{i,e}\mathfrak x_{i+2})(\mathbf T'_{i,e}\mathfrak d_{i+1})=-q^{-e}	\mathfrak m_{i+1}^{-e}\mathfrak d_{i}\mathfrak x_{i+2},\\
    &(\mathbf T'_{i,e}\mathfrak x_{i+1})(\mathbf T'_{i,e}\mathfrak d_{i+2})=-q^{e}	\mathfrak m_{i+1}^{e}\mathfrak x_{i}\mathfrak d_{i+2}.\label{eq:Txi+1Tdi+2}
\end{align}

From the braid relations in \eqref{eq:TTT}--\eqref{eq:TT}, one can show that
\begin{equation*}
\mathbf T'_{r,e}\mathfrak d_{r-2}=\mathfrak d_{r-2},\quad \mathbf T'_{r,e}\mathfrak x_{r-2}=\mathfrak x_{r-2}, \quad \mathbf T'_{i,e}\mathfrak x_{i+2}=\mathfrak x_{i+2}, \quad \mathbf T'_{i,e}\mathfrak d_{i+2}=\mathfrak d_{i+2}.
\end{equation*}

Therefore, by solving the system of equations \eqref{eq:Txr+1Tdr}--\eqref{eq:Txr-2Tdr-1} and \eqref{eq:Txi+1Tdi}--\eqref{eq:Txi+1Tdi+2}, we obtain
\begin{align*}
	&\mathbf T'_{i,e}(\mathfrak d_j)=\begin{cases}
		q^{e}\mathfrak m_r^{-2e}\mathfrak d_{r+1},  & \text{if } i=r,\ j=i+1,\\
		q^{-2e}\mathfrak m_r^{-e}\mathfrak m_{r+1}^{-e}\mathfrak d_r,  & \text{if } i=j=r,\\
		-q^{-e}\mathfrak m_{i+1}^{-e}\mathfrak d_i,  & \text{if } i\neq r,\ j=i+1,\\
		\mathfrak m_i^{-e}\mathfrak d_{i+1},  & \text{if } i=j\neq r,\\
		\mathfrak d_j,  & \text{otherwise},
	\end{cases}	
	\\
	&\mathbf T'_{i,e}(\mathfrak x_j)=\begin{cases}
		q^{-e}\mathfrak m_r^{2e}\mathfrak x_{r+1},  & \text{if } i=r,\ j=i+1,\\
		q^{e}\mathfrak m_r^{e}\mathfrak m_{r+1}^{e}\mathfrak x_r,  & \text{if } i=j=r,\\
		-q^{e}\mathfrak m_{i+1}^{e}\mathfrak x_i,  & \text{if } i\neq r,\ j=i+1,\\
		\mathfrak m_i^{e}\mathfrak x_{i+1},  & \text{if } i=j\neq r,\\
		\mathfrak x_j,  & \text{otherwise}.
	\end{cases}	
\end{align*}

Similarly, by \eqref{eq:TKxd}, we have
\begin{align*}
\mathbf	T'_{i,e}(\mathfrak m_j)=\begin{cases}
		\mathfrak m_i,  & \text{if } i\neq r,\ j=i+1,\\
		\mathfrak m_{i+1},  & \text{if } i\neq r,\ j=i,\\
		\mathfrak m_j,  & \text{otherwise}.
	\end{cases}	
\end{align*}

Hence $\mathbf T'_{i,e}$ coincide with the definition of $ T'_{i,e}$. By a similar approach, one can construct $\mathbf T''_{i,-e}: \mathbf{A}_q(\mathcal S)\to \mathbf{A}_q(\mathcal S)$ which is the same as $ T''_{i,-e}$. 

Summarizing the above, the theorem is proved.
\end{proof}

\section{Braid group action on polynomial ring}\label{sec:polynomial}
Let $\mathbb P:=\mathbb Q(q)[X_1,\cdots,X_{r},X_{r+1}]$ be a polynomial ring over $\mathbb Q(q)$.
\begin{thm}[{\cite[Theorem 3.2]{FGH}}]\label{thm:irreducible Aq module}
	The polynomial ring $\mathbb P$ is an irreducible $\mathbf{A}_q(\mathcal S)$-module with the following actions.
\begin{align*}\label{A_q action}
	{\mathfrak d_i}\mathbf X^{\mathbf a}=[(1+\delta_{i,\rho_n(r)}\delta_{i,r+1})a_i]\mathbf X^{\mathbf a-\mathbf {e}_i}
	,\quad{\mathfrak x_i}\mathbf X^{\mathbf a}=
	\mathbf X^{\mathbf a+\mathbf e_i}
	,\quad
	{\mathfrak m_i}\mathbf X^{\mathbf a}=
	q^{(1+\delta_{i,\rho_n(r)}\delta_{i,r+1})a_i}\mathbf X^{\mathbf a},
\end{align*}	
where $\mathbf X^{\mathbf a}=X_1^{a_1}\cdots X_{r+1}^{a_{r+1}}$ for $(a_1,\cdots,a_{r+1})\in\mathbb Z^{r+1}_{\geq 0}$ and $\mathbf e_i$ is the tuple such that the $i$-th element is $1$ and the other elements are $0$.
\end{thm}

Recall that $n=2r$ or $2r+1$.
We define linear operators $\mathcal T'_{i,e}, \mathcal T''_{i,-e}$ $(1\leq i\leq [\frac{n+1}{2}])$ on $\mathbb P$ as follows.
\begin{align*}
\mathcal T'_{i,e}(\mathbf X^{\mathbf a})&=
\begin{cases}
{(-q^{-e})}^{-a_{i+1}}q^{ea_ia_{i+1}}(s_i\mathbf X^{\mathbf a}),& \text{if } i\neq  r+1,\\
q^{\frac{1}{2}ea_r^2+\frac{3}{2}ea_r-ea_{r+1}+2ea_ra_{r+1}}\mathbf X^{\mathbf a},& \text{if } i= r,\ \rho_n(r)=r+1,\\
q^{\frac{1}{2}a_{r+1}^2e-\frac{1}{2}a_{r+1}e}\mathbf X^{\mathbf a},& \text{if } i=r+1,\ \rho_n(r+1)=r+1,\\
\end{cases}\\
\mathcal T''_{i,-e}(\mathbf X^{\mathbf a})&=
\begin{cases}
	{(-q^{-e})}^{a_{i}}q^{-ea_ia_{i+1}}(s_i\mathbf X^{\mathbf a}),& \text{if } i\neq r+1,\\
	q^{-\frac{1}{2}ea_r^2-\frac{3}{2}ea_r+ea_{r+1}-2ea_ra_{r+1}}\mathbf X^{\mathbf a},& \text{if } i= r,\ \rho_n(r)=r+1,\\
q^{-\frac{1}{2}a_{r+1}^2e+\frac{1}{2}a_{r+1}e}\mathbf X^{\mathbf a},& \text{if } i=r+1,\ \rho_n(r+1)=r+1,\\
\end{cases}
\end{align*}
where $s_i$ switch $X_i$ and $X_{i+1}$ in $\mathbf X^{\mathbf a}$.
\begin{lemma}\label{lem:calT'calT''=id}
The operators $\mathcal T'_{i,e}$ and $\mathcal T''_{i,-e}$ are inverse of each other, i.e.,
$$
\mathcal T'_{i,e}\mathcal T''_{i,-e}=\mathcal T''_{i,-e}\mathcal T'_{i,e}=\mathrm{id}.
$$
\end{lemma}
\begin{proof}
The assertion can be obtained by direct calculation.
\end{proof}
\begin{thm}\label{thm:Tkf}
	For any $\mathfrak k\in \mathbf A_q(\mathcal{S})$ and $f(X_1,\cdots,X_{r+1})\in\mathbb P$, we have 
\begin{equation}\label{eq:Ttf}
\mathcal T_{i}(\mathfrak k f(X_1,\cdots,X_{r+1}))=T_{i}(\mathfrak k)\mathcal T_{i}(f(X_1,\cdots,X_{r+1})),
\end{equation}
where $\mathcal T_i:=\mathcal T_{i,e}'$ or $\mathcal T_{i,-e}''$.
Moreover, we have the following braid relations.
\begin{align}
	&\mathcal T_{i-1}\mathcal T_i\mathcal T_{i-1}=\mathcal T_i\mathcal T_{i-1}\mathcal T_i, \quad \text{if } 2 \leq i \leq [\frac{n+1}{2}]-1,\label{Ti-1TiTi-1Xa}\\
	&\mathcal T_{i-1}\mathcal T_i\mathcal T_{i-1}\mathcal T_i=\mathcal T_i\mathcal T_{i-1}\mathcal T_i\mathcal T_{i-1}, \quad \text{if } i=[\frac{n+1}{2}],\label{Tr-1TrTr-1TrXa}\\
	&\mathcal T_i\mathcal T_j=\mathcal T_j\mathcal T_i, \quad \text{if } |i-j|\neq 1.\label{TiTjXa}
\end{align}
\end{thm}
\begin{proof}
We only show the proof for $\mathbf A_q^\jmath$.
If relation \eqref{eq:Ttf} holds for $\mathcal T_{i,e}'$, then by Lemma \ref{lem:calT'calT''=id}, we have
\begin{align*}
\mathcal T''_{i,-e}(\mathfrak k \mathbf X^{\mathbf a})&=\mathcal T''_{i,-e}\mathcal T''_{i,-e}\mathcal T'_{i,e}(\mathfrak k \mathbf X^{\mathbf a})=\mathcal T''_{i,-e}\mathcal T''_{i,-e}(T'_{i,e}(\mathfrak k) \mathcal T'_{i,e}(\mathbf X^{\mathbf a}))\\
&=\mathcal T''_{i,-e}\mathcal T''_{i,-e}(T'_{i,e}T'_{i,e}T''_{i,-e}(\mathfrak k) \mathcal T'_{i,e}\mathcal T'_{i,e}\mathcal T''_{i,-e}(\mathbf X^{\mathbf a}))\\
&=\mathcal T''_{i,-e}\mathcal T''_{i,-e}\mathcal T'_{i,e}\mathcal T'_{i,e}(T''_{i,-e}(\mathfrak k)T''_{i,-e}(\mathbf X^{\mathbf a}))=T''_{i,-e}(\mathfrak k)T''_{i,-e}(\mathbf X^{\mathbf a}).
\end{align*}	

Hence, for \eqref{eq:Ttf}, it is enough to show the following relations hold.
\begin{align}
	&\mathcal T'_{i,e}(\mathfrak d_j\mathbf X^{\mathbf a})=	 T'_{i,e}(\mathfrak d_j)	\mathcal T'_{i,e}(\mathbf X^{\mathbf a}),\label{TdX=TdTX}\\
	&\mathcal T'_{i,e}(\mathfrak x_j\mathbf X^{\mathbf a})=	 T'_{i,e}(\mathfrak x_j)	\mathcal T'_{i,e}(\mathbf X^{\mathbf a}),\label{TxX=TxTX}\\
	&\mathcal T'_{i,e}(\mathfrak m_j\mathbf X^{\mathbf a})=	 T'_{i,e}(\mathfrak m_j)	\mathcal T'_{i,e}(\mathbf X^{\mathbf a}).\label{TmX=TmTX}
\end{align}

For \eqref{TdX=TdTX}, if $i=r$, $j=r+1$, we have
\begin{align*}
&\mathcal T'_{r,e}(\mathfrak d_{r+1}\mathbf X^{\mathbf a})-T'_{r,e}(\mathfrak d_{r+1})\mathcal T'_{r,e}(\mathbf X^{\mathbf a})\\
=&[2a_{r+1}]q^{\frac{1}{2}ea_r^2+\frac{3}{2}ea_r-e(a_{r+1}-1)+2ea_r(a_{r+1}-1)}\mathbf (X^{\mathbf a-\mathbf e_{r+1}})\\
&-q^e\mathfrak m_r^{-2e}\mathfrak d_{r+1}q^{\frac{1}{2}ea_r^2+\frac{3}{2}ea_r-ea_{r+1}+2ea_ra_{r+1}}(\mathbf X^{\mathbf a})=0
\end{align*}

If $i=j=r$, we have
\begin{align*}
&\mathcal T'_{r,e}(\mathfrak d_{r}\mathbf X^{\mathbf a})-T'_{r,e}(\mathfrak d_{r})\mathcal T'_{r,e}(\mathbf X^{\mathbf a})\\
=&[a_{r}]q^{\frac{1}{2}e{(a_r-1)}^2+\frac{3}{2}e(a_r-1)-ea_{r+1}+2e(a_r-1)a_{r+1}}(\mathbf X^{\mathbf a-\mathbf e_{r}})\\
&-q^{-2e}\mathfrak m_r^{-e}\mathfrak m_{r+1}^{-e}\mathfrak d_rq^{\frac{1}{2}ea_r^2+\frac{3}{2}ea_r-ea_{r+1}+2ea_ra_{r+1}}(\mathbf X^{\mathbf a})=0.
\end{align*}

If $i=r$, $j\neq r$, $j\neq r+1$, we have
\begin{align*}
&\mathcal T'_{r,e}(\mathfrak d_{j}\mathbf X^{\mathbf a})-T'_{r,e}(\mathfrak d_{j})\mathcal T'_{r,e}(\mathbf X^{\mathbf a})\\
=&[a_{j}]q^{\frac{1}{2}ea_r^2+\frac{3}{2}ea_r-ea_{r+1}+2ea_ra_{r+1}}(\mathbf X^{\mathbf a-\mathbf e_{j}})\\
&-\mathfrak d_jq^{\frac{1}{2}ea_r^2+\frac{3}{2}ea_r-ea_{r+1}+2ea_ra_{r+1}}(\mathbf X^{\mathbf a})=0.
\end{align*}

If $i\neq r$, $j=i+1$, we have
\begin{align*}
&\mathcal T'_{i,e}(\mathfrak d_{i+1}\mathbf X^{\mathbf a})-T'_{i,e}(\mathfrak d_{i+1})\mathcal T'_{i,e}(\mathbf X^{\mathbf a})\\
=&[a_{i+1}]{(-q^{-e})}^{-(a_{i+1}-1)}q^{ea_i(a_{i+1}-1)}((s_i\mathbf X^{\mathbf a})\mathbf X^{-\mathbf e_i})\\
&+q^{-e}\mathfrak m_{i+1}^{-e}\mathfrak d_i{(-q^{-e})}^{-a_{i+1}}q^{ea_ia_{i+1}}(s_i\mathbf X^{\mathbf a})=0.
\end{align*}

If $i\neq r$, $i=j$, we have
\begin{align*}
&\mathcal T'_{i,e}(\mathfrak d_{i}\mathbf X^{\mathbf a})-T'_{i,e}(\mathfrak d_{i})\mathcal T'_{i,e}(\mathbf X^{\mathbf a})\\
=&[a_{i}]{(-q^{-e})}^{-a_{i+1}}q^{ea_{i+1}(a_i-1)}((s_i\mathbf X^{\mathbf a})\mathbf X^{-\mathbf e_{i+1}})\\
&-\mathfrak m_{i}^{-e}\mathfrak d_{i+1}{(-q^{-e})}^{-a_{i+1}}q^{ea_ia_{i+1}}(s_i\mathbf X^{\mathbf a})=0.
\end{align*}

If $i\neq r$, $j\neq i$, $j\neq i+1$, we have
\begin{align*}
&\mathcal T'_{i,e}(\mathfrak d_{j}\mathbf X^{\mathbf a})-T'_{i,e}(\mathfrak d_{j})\mathcal T'_{i,e}(\mathbf X^{\mathbf a})\\
=&[a_{j}]{(-q^{-e})}^{-a_{i+1}}q^{ea_{i+1}a_i}(\mathbf X^{\mathbf a-\mathbf e_j})\\
&-{(-q^{-e})}^{-a_{i+1}}q^{ea_{i+1}a_i}[a_{j}](\mathbf X^{\mathbf a-\mathbf e_j})=0.
\end{align*}

For \eqref{TxX=TxTX}, if $i=r$, $j=r+1$, we have
\begin{align*}
	&\mathcal T'_{r,e}(\mathfrak x_{r+1}\mathbf X^{\mathbf a})-T'_{r,e}(\mathfrak x_{r+1})\mathcal T'_{r,e}(\mathbf X^{\mathbf a})\\
	=&q^{\frac{1}{2}ea_r^2+\frac{3}{2}ea_r-e(a_{r+1}+1)+2ea_r(a_{r+1}+1)}(\mathbf X^{\mathbf a+\mathbf e_{r+1}})\\
	&-q^{\frac{1}{2}ea_r^2+\frac{7}{2}ea_r-ea_{r+1}+2ea_ra_{r+1}-e}(\mathbf X^{\mathbf a+\mathbf e_{r+1}})=0.
\end{align*}

If $i=j=r$, we have
\begin{align*}
	&\mathcal T'_{r,e}(\mathfrak x_{r}\mathbf X^{\mathbf a})-T'_{r,e}(\mathfrak x_{r})\mathcal T'_{r,e}(\mathbf X^{\mathbf a})\\
	=&q^{\frac{1}{2}e{(a_r+1)}^2+\frac{3}{2}e(a_r+1)-ea_{r+1}+2e(a_r+1)a_{r+1}}(\mathbf X^{\mathbf a+\mathbf e_{r}})\\
	&-q^{\frac{1}{2}ea_r^2+\frac{5}{2}ea_r+ea_{r+1}+2ea_ra_{r+1}+2e}(\mathbf X^{\mathbf a+\mathbf e_{r}})=0.
\end{align*}

If $i=r$, $j\neq r$, $j\neq r+1$ we have
\begin{align*}
	&\mathcal T'_{r,e}(\mathfrak x_{j}\mathbf X^{\mathbf a})-T'_{r,e}(\mathfrak x_{r})\mathcal T'_{r,e}(\mathbf X^{\mathbf a})\\
	=&q^{\frac{1}{2}e{(a_r+1)}^2+\frac{3}{2}e(a_r+1)-ea_{r+1}+2e(a_r+1)a_{r+1}}(\mathbf X^{\mathbf a+\mathbf e_{r}})\\
	&-q^{\frac{1}{2}ea_r^2+\frac{5}{2}ea_r+ea_{r+1}+2ea_ra_{r+1}+2e}(\mathbf X^{\mathbf a+\mathbf e_{r}})=0.
\end{align*}

If $i\neq r$, $j=i+1$, we have
\begin{align*}
	&\mathcal T'_{i,e}(\mathfrak x_{i+1}\mathbf X^{\mathbf a})-T'_{i,e}(\mathfrak x_{i+1})\mathcal T'_{i,e}(\mathbf X^{\mathbf a})\\
	=&{(-q^{-e})}^{-a_{i+1}-1}q^{ea_i(a_{i+1}+1)}((s_i\mathbf X^{\mathbf a})\mathbf X^{\mathbf e_i})\\
	&+q^{e}\mathfrak m^{e}_{i+1}\mathfrak x_i{(-q^{-e})}^{-a_{i+1}}q^{ea_ia_{i+1}}(s_i\mathbf X^{\mathbf a})=0.
\end{align*}

If $i\neq r$, $j=i$, we have
\begin{align*}
	&\mathcal T'_{i,e}(\mathfrak x_{i}\mathbf X^{\mathbf a})-T'_{i,e}(\mathfrak x_{i})\mathcal T'_{i,e}(\mathbf X^{\mathbf a})\\
	=&{(-q^{-e})}^{-a_{i+1}}q^{ea_{i+1}(a_{i}+1)}((s_i\mathbf X^{\mathbf a})\mathbf X^{\mathbf e_{i+1}})\\
	&-\mathfrak m^{e}_{i}\mathfrak x_{i+1}{(-q^{-e})}^{-a_{i+1}}q^{ea_ia_{i+1}}(s_i\mathbf X^{\mathbf a})=0.
\end{align*}

If $i\neq r$, $j\neq i$, $j\neq i+1$ we have
\begin{align*}
	&\mathcal T'_{i,e}(\mathfrak x_{j}\mathbf X^{\mathbf a})-
	T'_{i,e}(\mathfrak x_{j})\mathcal T'_{i,e}(\mathbf X^{\mathbf a})\\
	=&{(-q^{-e})}^{-a_{i+1}}q^{ea_{i+1}a_{i}}(s_i\mathbf X^{\mathbf a+\mathbf e_j})
	-\mathfrak x_j{(-q^{-e})}^{-a_{i+1}}q^{ea_ia_{i+1}}(s_i\mathbf X^{\mathbf a})
	=0.
\end{align*}

To sum up the above calculations, for all cases, the relations \eqref{TdX=TdTX} and \eqref{TxX=TxTX} hold.
Similar to the proof above, we can prove the formula \eqref{TmX=TmTX} is true.
\vskip 2mm
For the braid relations \eqref{Ti-1TiTi-1Xa}--\eqref{TiTjXa}, if $2 \leq i \leq r-1$, we have
\begin{align*}
\mathcal T'_{i,e}\mathbf X^{\mathbf a}
&={(-q^{-e})}^{-a_{i+1}}q^{ea_ia_{i+1}}(s_i\mathbf X^{\mathbf a}),\\
\mathcal T'_{i-1,e}\mathcal T'_{i,e}\mathbf X^{\mathbf a}
&={(-q^{-e})}^{-2a_{i+1}}q^{ea_ia_{i+1}+ea_{i-1}a_{i+1}}(s_{i-1}s_i\mathbf X^{\mathbf a}),\\
\mathcal T'_{i,e}\mathcal T'_{i-1,e}\mathcal T'_{i,e}\mathbf X^{\mathbf a}
&={(-q^{-e})}^{-2a_{i+1}-a_i}q^{ea_ia_{i+1}+ea_{i-1}a_{i+1}+ea_{i-1}a_i}(s_is_{i-1}s_i\mathbf X^{\mathbf a}),\\
\mathcal T'_{i-1,e}\mathbf X^{\mathbf a}
&={(-q^{-e})}^{-a_{i}}q^{ea_{i-1}a_{i}}(s_{i-1}\mathbf X^{\mathbf a}),\\
\mathcal T'_{i,e}\mathcal T'_{i-1,e}\mathbf X^{\mathbf a}
&={(-q^{-e})}^{-a_i-a_{i+1}}q^{ea_ia_{i-1}+ea_{i-1}a_{i+1}}(s_is_{i-1}\mathbf X^{\mathbf a}),\\
\mathcal T'_{i-1,e}\mathcal T'_{i,e}\mathcal T'_{i-1,e}\mathbf X^{\mathbf a}
&={(-q^{-e})}^{-2a_{i+1}-a_i}q^{ea_ia_{i+1}+ea_{i-1}a_{i+1}+ea_{i-1}a_i}(s_{i-1}s_is_{i-1}\mathbf X^{\mathbf a}).
\end{align*}

Since $s_is_{i-1}s_i\mathbf X^{\mathbf a}=s_{i-1}s_is_{i-1}\mathbf X^{\mathbf a}$, the relation $\mathcal T'_{i,e}\mathcal T'_{i-1,e}\mathcal T'_{i,e}=\mathcal T'_{i-1,e}\mathcal T'_{i,e}\mathcal T'_{i-1,e}$ holds.
\vskip 2mm
If $i=r$, we compute $\mathcal T'_{r-1,e}\mathcal T'_{r,e}\mathcal T'_{r-1,e}(\mathbf X^{\mathbf a})$ as follows.
\begin{equation}\label{eq:Tr-1TrTr-1}
\begin{aligned}
&\mathcal T'_{r-1,e}\mathcal T'_{r,e}\mathcal T'_{r-1,e}(\mathbf X^{\mathbf a})\\
=&\mathcal T'_{r-1,e}\mathcal T'_{r,e}{(-q^{-e})}^{-a_r}q^{ea_{r-1}a_r}(s_{r-1}\mathbf X^{\mathbf a})\\
=&q^{\frac{1}{2}ea_{r-1}^2+\frac{3}{2}ea_{r-1}-ea_{r+1}+2ea_{r-1}a_{r+1}}{(-q^{-e})}^{-a_r}q^{ea_{r-1}a_r}\mathcal T'_{r-1,e}(s_{r-1}\mathbf X^{\mathbf a})\\
=&q^{\frac{1}{2}ea_{r-1}^2+\frac{3}{2}ea_{r-1}-ea_{r+1}+2ea_{r-1}a_{r+1}}{(-q^{-e})}^{a_{r-1}-a_r}q^{2ea_{r-1}a_r}(\mathbf X^{\mathbf a}).
\end{aligned}
\end{equation}

By \eqref{eq:Tr-1TrTr-1}, we have
\begin{align*}
	&\mathcal T'_{r,e}\mathcal T'_{r-1,e}\mathcal T'_{r,e}\mathcal T'_{r-1,e}(\mathbf X^{\mathbf a})\\
	=&q^{\frac{1}{2}ea_{r-1}^2+\frac{3}{2}ea_{r-1}-ea_{r+1}+2ea_{r-1}a_{r+1}}{(-q^{-e})}^{a_{r-1}-a_r}q^{2ea_{r-1}a_r}\mathcal T'_{r,e}(\mathbf X^{\mathbf a})
\end{align*}
and
\begin{align*}
	&\mathcal T'_{r-1,e}\mathcal T'_{r,e}\mathcal T'_{r-1,e}(\mathcal T'_{r,e}\mathbf X^{\mathbf a})\\
	=&q^{\frac{1}{2}ea_{r-1}^2+\frac{3}{2}ea_{r-1}-ea_{r+1}+2ea_{r-1}a_{r+1}}{(-q^{-e})}^{a_{r-1}-a_r}q^{2ea_{r-1}a_r}(\mathcal T'_{r,e}\mathbf X^{\mathbf a}).
\end{align*}

Hence, we have $\mathcal T'_{r,e}\mathcal T'_{r-1,e}\mathcal T'_{r,e}\mathcal T'_{r-1,e}=\mathcal T'_{r-1,e}\mathcal T'_{r,e}\mathcal T'_{r-1,e}\mathcal T'_{r,e}$.
The proof of \eqref{TiTjXa} is trivial, so we omit it.
\end{proof}

\begin{thm}
The operator $\mathcal T_i$ satisfying \eqref{eq:Ttf} is unique.
\end{thm}

\begin{proof}
We show the proof for the action of $\mathbf A_q^\jmath$ on $\mathbb P$.
We assume that there exists another linear operators $\mathfrak T'_{i,e}$ such that 	\eqref{eq:Ttf} holds. For $1 \leq i\leq r$ and $\mathbf a=(a_1,a_2,\cdots,a_{r+1})\in\mathbb Z_{\geq 0}^{r+1}$, we consider the following action.
\begin{equation}\label{eq:frakTdX}
\mathfrak T'_{i,e}((\mathfrak d_1^{a_1}\cdots\mathfrak d_{r+1}^{a_{r+1}})\mathbf X^{\mathbf a})=T_{i,e}(\mathfrak d_1^{a_1}\cdots\mathfrak d_{r+1}^{a_{r+1}})\mathfrak T'_{i,e}(\mathbf X^{\mathbf a}).
\end{equation}

If $i=r$, then the formula in \eqref{eq:frakTdX} will be written in the following form.
\begin{equation}\label{eq:TdXreduce}
\begin{aligned}
(\prod_{i=1}^{r}[a_i]!)[a_{r+1}]!!&=(\prod_{i=1}^{r-1}\mathfrak d_i^{a_i}){(q^{-2e}\mathfrak m_r^{-e}\mathfrak m_{r+1}^{-e}\mathfrak d_r)}^{a_r}{(q^e\mathfrak m_r^{-2e}\mathfrak d_{r+1})}^{a_{r+1}}\mathfrak T'_{r,e}(\mathbf X^{\mathbf a})\\
&=(\prod_{i=1}^{r-1}\mathfrak d_i^{a_i}) 
q^{-2ea_r}\mathfrak m_{r+1}^{-ea_r}{(\mathfrak m_{r}^{-e}\mathfrak d_r)}^{a_r}q^{ea_{r+1}}\mathfrak m_r^{-2ea_{r+1}}\mathfrak d_{r+1}^{a_{r+1}}
\mathfrak T'_{r,e}(\mathbf X^{\mathbf a})\\
&=q^{ea_{r+1}-2ea_r-e\frac{a_r(a_r-1)}{2}-2ea_ra_{r+1}}
\mathfrak m_{r+1}^{-ea_r}\mathfrak m_r^{-ea_r-2ea_{r+1}}(\prod_{i=1}^{r+1}\mathfrak d_i^{a_i}) 
\mathfrak T'_{r,e}(\mathbf X^{\mathbf a})\\
&=q^{-\frac{1}{2}ea_r^2-\frac{3}{2}ea_r+ea_{r+1}-2ea_ra_{r+1}}
\mathfrak m_{r+1}^{-ea_r}\mathfrak m_r^{-ea_r-2ea_{r+1}}(\prod_{i=1}^{r+1}\mathfrak d_i^{a_i}) 
\mathfrak T'_{r,e}(\mathbf X^{\mathbf a})\\
\end{aligned}
\end{equation}
where
$[a_{r+1}]!!=[2a_{r+1}][2a_{r+1}-2]\cdots[2]$.
\vskip 2mm
On the other hand, we have the following formula.
\begin{equation}\label{eq:qmdX}
\begin{aligned}
&q^{-\frac{1}{2}ea_r^2-\frac{3}{2}ea_r+ea_{r+1}-2ea_ra_{r+1}}
\mathfrak m_{r+1}^{-ea_r}\mathfrak m_r^{-ea_r-2ea_{r+1}}(\prod_{i=1}^{r+1}\mathfrak d_i^{a_i})\mathbf X^{\mathbf a}\\
=&q^{-\frac{1}{2}ea_r^2-\frac{3}{2}ea_r+ea_{r+1}-2ea_ra_{r+1}}
 (\prod_{i=1}^{r}[a_i]!)[a_{r+1}]!!
\end{aligned}
\end{equation}

Comparing \eqref{eq:TdXreduce} and \eqref{eq:qmdX}, we have
\begin{equation}\label{eq:frakTrXa}
\mathfrak T'_{r,e}(\mathbf X^{\mathbf a})=q^{\frac{1}{2}ea_r^2+\frac{3}{2}ea_r-ea_{r+1}+2ea_ra_{r+1}}\mathbf X^{\mathbf a}.
\end{equation}
\vskip 2mm
If $i\neq r$, the formula in \eqref{eq:frakTdX} will be written in the following form.
\begin{equation}\label{eq:TdXreduce2}
	\begin{aligned}
		(\prod_{j=1}^{r}[a_j]!)[a_{r+1}]!!&=(\prod_{j=1}^{i-1}\mathfrak d_j^{a_j}){(\mathfrak m_i^{-e}\mathfrak d_{i+1})}^{a_i}{(-q^{-e}\mathfrak m_{i+1}^{-e}\mathfrak d_{i})}^{a_{i+1}}
		(\prod_{j=i+2}^{r+1}\mathfrak d_j^{a_j})
		\mathfrak T'_{i,e}(\mathbf X^{\mathbf a})\\
		&={(-q^{-e})}^{a_{i+1}}q^{-ea_ia_{i+1}}\mathfrak m_i^{-ea_i}\mathfrak m_{i+1}^{-ea_{i+1}}(\prod_{j=1}^{i-1}\mathfrak d_j^{a_j})
		\mathfrak d_i^{a_{i+1}}\mathfrak d_{i+1}^{a_{i}}
		(\prod_{j=i+2}^{r+1}\mathfrak d_j^{a_j})	
		\mathfrak T'_{i,e}(\mathbf X^{\mathbf a})
		\\
	\end{aligned}
\end{equation}

On the other hand, we have
\begin{equation}\label{eq:qmdX2}
	\begin{aligned}
		&{(-q^{-e})}^{a_{i+1}}q^{-ea_ia_{i+1}}\mathfrak m_i^{-ea_i}\mathfrak m_{i+1}^{-ea_{i+1}}(\prod_{j=1}^{i-1}\mathfrak d_j^{a_j})
		\mathfrak d_i^{a_{i+1}}\mathfrak d_{i+1}^{a_{i}}
		(\prod_{j=i+2}^{r+1}\mathfrak d_j^{a_j})(s_i\mathbf X^{\mathbf a})\\
		=&{(-q^{-e})}^{a_{i+1}}q^{-ea_ia_{i+1}}
		(\prod_{i=1}^{r}[a_i]!)[a_{r+1}]!!
	\end{aligned}
\end{equation}

Comparing \eqref{eq:TdXreduce2} and \eqref{eq:qmdX2}, we have
\begin{equation}\label{eq:frakTiXa}
\mathfrak T'_{i,e}(\mathbf X^{\mathbf a})=	{(-q^{-e})}^{-a_{i+1}}q^{ea_ia_{i+1}}(s_i\mathbf X^{\mathbf a}).
	\end{equation}

Combining the formulas in \eqref{eq:frakTrXa} and \eqref{eq:frakTiXa}, we can see that $\mathfrak T'_{i,e}$ coincide with $\mathcal T'_{i,e}$. By a similar approach, we can also prove that the linear operator $\mathcal T''_{i,-e}$ is unique.
	\end{proof}

By Proposition \ref{prop:UjtoAq} and Theorem \ref{thm:irreducible Aq module}, the polynomial ring $\mathbb P$ is a $^\imath\bfU(\mathcal{S})$-module with the following actions.
\begin{equation*}
B_i\mathbf X^{\mathbf a}=\varphi(B_i)\mathbf X^{\mathbf a},\quad K_i\mathbf X^{\mathbf a}=\varphi(K_i)\mathbf X^{\mathbf a}.
\end{equation*}

\begin{thm}
	For any $u\in {^\imath\bfU(\mathcal{S})}$ and $f(X_1,\cdots,X_{r+1})\in\mathbb P$, we have 
	\begin{equation*}
		\mathcal T_{i}(u f(X_1,\cdots,X_{r+1}))=\tau_{i}(u)\mathcal T_{i}(f(X_1,\cdots,X_{r+1})),
	\end{equation*}
	where $\mathcal T_i:=\mathcal T_{i,e}'$ or $\mathcal T_{i,-e}''$.

\end{thm}
\begin{proof}
It is enough to show the following relations hold.
\begin{align}
	&\mathcal T'_{i,e}(B_j \mathbf X^{\mathbf a})=\tau'_{i,e}(B_j)\mathcal T'_{i,e}(\mathbf X^{\mathbf a}),\label{TeX=taueTX}\\
	&\mathcal T'_{i,e}(K_j \mathbf X^{\mathbf a})=\tau'_{i,e}(K_j)\mathcal T'_{i,e}(\mathbf X^{\mathbf a}).\label{TkX=taukTX}
\end{align}

We only show the proof of \eqref{TeX=taueTX} for $\bfU^\jmath$.
If $1\leq j\leq r$, by Theorem \ref{thm:Tkf} and Theorem \ref{thm:Tphi=phitau}, we have
\begin{align*}
\mathcal T'_{i,e}(B_j\mathbf X^{\mathbf a})
&=\mathcal T'_{i,e}(\varphi(B_j)\mathbf X^{\mathbf a})=\mathcal T'_{i,e}(\mathfrak x_j\mathfrak d_{j+1}\mathbf X^{\mathbf a})=T'_{i,e}(\mathfrak x_j\mathfrak d_{j+1})\mathcal T'_{i,e}(\mathbf X^{\mathbf a})\\
&=T'_{i,e}(\varphi(B_j))\mathcal T'_{i,e}(\mathbf X^{\mathbf a})=\varphi(\tau'_{i,e}(B_j))\mathcal T'_{i,e}(\mathbf X^{\mathbf a})=\tau'_{i,e}(B_j)\mathcal T'_{i,e}(\mathbf X^{\mathbf a}).
\end{align*}

The proof for the case $\rho_n(r)\leq j\leq \rho_n(1)$ is similar.
\end{proof}

The commutative diagram below shows the relations between $\tau_i$, $T_i$ and $\mathcal T_i$.
\begin{center}
\begin{tikzcd}
	^\imath\bfU(\mathcal{S}) \arrow[r, "\varphi"] \arrow[d, "\tau_i"] & \mathbf A_q(\mathcal{S}) \arrow[r, "\circlearrowleft"] \arrow[d, "T_i"] & \mathbb P \arrow[d, "\mathcal T_i"] \\
	^\imath\bfU(\mathcal{S}) \arrow[r, "\varphi"']          & \mathbf A_q(\mathcal{S}) \arrow[r, "\circlearrowleft"]                                         & \mathbb P
\end{tikzcd}
\end{center}


\end{document}